\documentclass[11pt, twoside, a4paper]{article}
\usepackage[utf8]{inputenc}
\usepackage[english]{babel}
\usepackage[title]{appendix}

\usepackage{amssymb,amsmath,amsthm}
\usepackage{bigints}

\usepackage{dsfont}
\usepackage[mathscr]{eucal} 
\usepackage{graphicx}

\usepackage{xcolor}
\usepackage[hidelinks]{hyperref}
\hypersetup{
 colorlinks,
 linkcolor={red!50!black},
 citecolor={blue!50!black},
 urlcolor={blue!80!black}
}
\usepackage{pdfsync}

\usepackage{enumerate} 
\usepackage{esint}

\theoremstyle{plain}
\newtheorem{theorem}{Theorem}[section]
\newtheorem{definition}[theorem]{Definition}

\newtheorem{lemma}[theorem]{Lemma}

\newtheorem{proposition}[theorem]{Proposition}
\theoremstyle{remark}
\newtheorem{remark}[theorem]{Remark}

\numberwithin{equation}{section}
\allowdisplaybreaks

\def\R{{\mathbb R}}
\def\EE{{\mathcal E}}
\def\RR{{\mathcal R}}

\newcommand{\dd}{\,\mathrm{d}}

\renewcommand{\div}{\operatorname{div}}

\renewcommand{\leq}{\leqslant}\renewcommand{\le}{\leqslant}
\renewcommand{\geq}{\geqslant}\renewcommand{\ge}{\geqslant}
\renewcommand{\epsilon}{\varepsilon}

\newcommand{\dt}{\partial_t}
\newcommand{\eps}{\varepsilon}

\newcommand{\ap}{\alpha_+}
\newcommand{\am}{\alpha_-}
\newcommand{\rp}{\rho_+}
\newcommand{\rmi}{\rho_-}
\newcommand{\gp}{\gamma_{+}}
\newcommand{\gm}{\gamma_-}

\title{Low-Mach-number limit of a compressible two-phase flow system with algebraic closure} 
\author{Cassandre Lebot}
\def\adrese{
\noindent CNRS, LAMA, ISTerre, Univ. Savoie Mont Blanc, 73000 Chamb\'ery, France.\\
{\it Email address:} \texttt{Cassandre.Lebot@univ-smb.fr}
}

\date{\today}

\begin{document}
\maketitle

\begin{abstract}
We analyse a bi-fluid isentropic compressible Navier-Stokes system with barotropic pressure laws in a two-phase framework with equal pressure and single velocity. We focus on the rigorous analysis of the low Mach number limit under well-prepared initial data.
Our main result shows that, as the Mach number tends to zero, the partial densities converge to constant states while the velocity field converges to a divergence-free vector field, and we recover the incompressible non-homogenous fluid system. The volume fractions remain nontrivial and are transported by the limit flow.

Our method relies on the introduction of suitable modulated quantities and on two relative entropy functionals adapted to the two-phase structure: a standard entropy commonly used in the literature, and a logarithmic entropy, which is essential here as the former is not sufficient due to the structure of the underlying two-phase system.

\end{abstract}

\tableofcontents

\section{Introduction}

This paper addresses the rigorous justification of the low-Mach-number limit for a compressible two-phase flow system. The model under consideration is a compressible isentropic Navier-Stokes system describing the evolution of a mixture of two fluids sharing a common velocity and pressure field, with their densities governed by coupled mass balance equations under barotropic pressure laws.

Physically, the study of two-phase compressible flows in the low Mach number regime is motivated by a wide range of industrial and geophysical applications, where fluids mixtures move at speeds much lower than the speed of sound. Mathematically, the problem provides a natural extension of the classical incompressible limit for single-phase compressible Navier–Stokes equations. However, the presence of several phases introduces additional variables and couplings, resulting in a richer and more complex structure than in the single-phase setting.

In most modeling frameworks for such mixtures, each phase satisfies its own system of balance laws, supplemented by closure relations that account for thermodynamic quantities and interfacial interactions. Classical examples include the two–fluid models developed by Ishii and Hibiki \cite{ishii2010thermo} and the hyperbolic system introduced by Baer and Nunziato \cite{baer1986two}. Depending on the modeling assumptions, these closures may introduce additional evolution equations or take the form of algebraic relations linking pressure, densities and volume fractions.

The presence of several phases introduces new challenges: the system typically involves several densities and volume fractions transported by the flow, interacting through the pressure and momentum equations. This additional structure makes the rigorous justification of the low-Mach-number limit substantially more delicate, as the coupling between phases complicates the derivation of uniform estimates and the control of nonlinear interactions.

\subsection{Main result}

For our study we consider a compressible isentropic two-phase flow Navier-Stokes system with algebraic closure. Let us denote by $\eps$ the Mach number. After nondimensionalization, see Section~\ref{nondim} below for details, denoting the phases by $+$ and $-$, the system reads as
\begin{equation}
 \left\{
 \begin{aligned}\label{system_eps}
 &\partial_t (\alpha_+^{\epsilon} \rho_+^{\epsilon}) + \div(\alpha_+^{\epsilon}\rho_+^{\epsilon}u^{\epsilon}) = 0,\\
 &\partial_t (\alpha_-^{\epsilon} \rho_-^{\epsilon}) + \div(\alpha_-^{\epsilon}\rho_-^{\epsilon}u^{\epsilon}) = 0,\\
 &\partial_t ((\alpha_+^{\epsilon} \rho_+^{\epsilon} + \alpha_-^{\epsilon}\rho_-^{\epsilon})u^{\epsilon}) + \div((\alpha_+^{\epsilon}\rho_+^{\epsilon} + \alpha_-^{\epsilon}\rho_-^{\epsilon})u^{\epsilon}\otimes u^{\epsilon})+ \frac{1}{\epsilon^2} \nabla \Big(\ap^\eps p_+(\rho_+^{\epsilon}) + \am^\eps p_-(\rmi^\eps)\Big) \\
 &\hspace{7cm} = \mu \Delta u^{\epsilon} + (\mu + \lambda)\nabla \div u^{\epsilon},\\
 &p_i(\rho_i^{\epsilon}) = (\rho_i^{\epsilon})^{\gamma_i},\qquad \rho_i^\epsilon \geq 0, \qquad i=+,-\\
 &p_+(\rho_+^{\epsilon}) = p_-(\rho_-^{\epsilon}),\\
 &\alpha_+^{\epsilon} + \alpha_-^{\epsilon} = 1,\qquad 0\leq \ap^\epsilon \leq 1,
 \end{aligned}
 \right .
\end{equation}
where $\rho_i^\eps$, $p_i$, $\alpha_i^\eps$ are the density, pressure and volume fraction of phase $i$ and we consider one single velocity $u^\eps$. The real numbers $\mu$ and $\lambda$ are the constant shear and bulk viscosities verifying, as usual,
\[
\mu >0 \quad \text{and} \quad \lambda + 2 \mu>0.
\]
The system is endowed with the Dirichlet boundary condition
\begin{equation}\label{eq.bc}
 u^\eps(t,x)\lvert_{\partial \Omega} = 0
\end{equation}
on $(0,T) \times \partial \Omega$, where $\Omega$ is a sufficiently smooth bounded domain of $\mathbb{R}^3$ and $T>0$ is arbitrary large.

As usual in the mathematical studies of bi-fluid models, we introduce the fractional masses $R_\pm^\eps := \alpha_\pm^\eps \rho_\pm^\eps$, the mixture density by $\rho^\eps := R_+^\eps + R_-^\eps$ and the momentum by $m^\eps:=\rho^\eps u^\eps$. As emphasized in \cite{BreschHuangLi}, it is important to be able to pass back and forth between the variables $(\rho_{+}^\eps,R_{+}^\eps, m^\epsilon)$ and $(\alpha_\pm^\eps, \rho_\pm^\eps,u^\eps)$.

\begin{definition}\label{rem.compa}
 For $0\leq R_+^\eps \leq \rho_+^\eps$ and $m^\epsilon$ given, if $\rho_+^\eps>0$, we may define $\alpha_+^\eps=R_+^\eps/\rho_+^\eps\in [0,1]$. Otherwise, we set $\alpha_+^\eps=1/2$ on $\{\rho_+^\eps=0\}$. Next we define $\alpha_-^\eps=1-\alpha_+^\eps\in[0,1]$, $\rho_-^\eps =(\rho_+^\eps)^{\gamma_+/\gamma_-}\geq 0$ (by the equality of pressures), $R_-^\epsilon=\alpha_-^\eps \rho_-^\eps$ and $\rho^\eps = R_+^\eps+R_-^\epsilon$. If $\rho_+^\eps>0$, we get that $\rho_-^\eps,\rho^\eps>0$ and we may define the velocity $u^\eps = m^\eps/\rho^\eps$. Otherwise, we set $u^\eps=0$ on $\{\rho_+^\eps=0\}$.

Any functions $(\rho_{+,0}^\eps,R_{+,0}^\eps, m_0^\epsilon)\in L^{\gamma_+}(\Omega)\times L^{\gamma_+}(\Omega)\times L^{1}(\Omega)$ are compatibles for the bi-fluid compressible system if
\begin{equation}\label{compatibilityeps}
 0\leq R_{+,0}^\eps \leq \rho_{+,0}^\eps\text{\ a.e.}, \quad \sqrt{\rho_0^\eps} u_0^\eps\in L^2(\Omega),
\end{equation}
where we have defined $\alpha_{\pm,0}^\eps\in [0,1]$, $\rho_{-,0}^\eps\geq 0$, $\rho^\eps_0\geq 0$ and $u_0^\eps$ from $(\rho_{+,0}^\eps,R_{+,0}^\eps, m_0^\epsilon)$.
 \end{definition}
For such compatible initial data, we consider a global weak solution ``à la Leray'', see Definition~\ref{defp} for the precise meaning.

System~\eqref{system_eps} is a particular case of the system studied by \cite{novotny2020weak} for which they obtain the existence of global-in-time finite-energy weak solutions for large initial data, but adding some conditions on $\gamma_\pm$. Note that \cite{bresch2019finite} also gives a proof of the existence of global-in-time weak solutions for a semi-stationary bifluid system when $\gamma_+\neq \gamma_-$.

Let us now investigate the Mach-number limit. As shown in \cite{lebot2025low}, when we let the Mach number go to zero in \eqref{system_eps}, we obtain formally the following non-homogeneous viscous incompressible Navier-Stokes equations:
\begin{equation}\label{limitsys}
 \begin{cases}
 \rp = C_0^{\frac{1}{\gp}},\qquad \rmi = C_0^{\frac{1}{\gm}},\\
 \div u = 0,\\
 \dt \ap + (u \cdot \nabla) \ap =0,\\
 \left(\rp \ap + \rmi \am \right) \left( \dt u + (u \cdot \nabla ) u\right) + \nabla \Pi = \mu \Delta u ,\\
 \ap + \am = 1,
 \end{cases}
\end{equation}
where $C_0>0$ is a real number and $u$ satisfies the Dirichlet boundary condition \eqref{eq.bc}.
Let us notice that the first line implies the pressure equality: 
\[
p_+(\rp)=p_-(\rmi).
\]
The densities $\rho_\pm$ are now constant, and the continuity equations on $(\alpha_\pm^\eps\rho_\pm^\eps)$ became a transport equation on $\alpha_\pm$ by a divergence free vector field. It is then clear that the initial assumption $\alpha_{\pm,0}\in [0,1]$ remains true for all time $t>0$.

Denoting again the mixture density by $\rho:=\rp \ap + \rmi \am$ we get the classical non-homogeneous viscous incompressible Navier-Stokes equations where $\rho$ satisfies the transport equation by $u$.

The existence of strong solutions to this system has been extensively studied in the literature (see Section~\ref{sec.incomp}). For completeness, we recall later an existence result (Theorem~\ref{sol-lim}) taken from \cite{salvi1991equations}, which specifies the notion of strong solution considered in this work.

In the next theorem, we need to define the notion of well prepared data, which is related to the entropies used. Throughout the article, we will use the following quantities:
\begin{itemize}
 \item the energy
 \begin{equation}\label{def-Energy}
 E^\eps(t) := \frac{1}{2} \int_{\Omega} \Big(\rho^\varepsilon |u^\eps|^2\Big)(t,x) \dd x
 + \sum_{i=+,-} \int_{\Omega} \frac{\alpha_i^{\epsilon}(t,x)}{\epsilon^2} H_i(\rho_i^\eps | \rho_i)(t,x)\dd x ,
 \end{equation}
 with, for $i=+,-$,
\begin{equation}\label{H}
 H_i(\rho_i^\eps | \rho_i) := \frac{1}{\gamma_i-1} \left( (\rho_i^{\epsilon})^{\gamma_i} - \rho_i^{\gamma_i} - \gamma_i\rho_i^{\gamma_i-1}(\rho_i^{\epsilon}-\rho_i) \right).
\end{equation}
 \item the natural modulated entropy
 \begin{equation}\label{def-E1}
 \begin{aligned}
 \EE_1^\eps(t) &= \EE_1\Big(\rho^{\epsilon}_+,\rho^{\epsilon}_-,\alpha^{\epsilon}_+,u^{\epsilon}| \rho_+,\rho_-,\alpha_+,u\Big) \\
 &:=
 \frac{1}{2} \int_{\Omega} \rho^{\epsilon}(t,x) |u^{\epsilon} - u|^2(t,x)\dd x 
 + \sum_{i=+,-} \int_{\Omega} \frac{\alpha_i^{\epsilon}(t,x)}{\epsilon^2} H_i(\rho_i^\eps | \rho_i)(t,x)\dd x .
 \end{aligned}
 \end{equation}
 \item the Kullback type relative entropy
 \begin{equation}\label{def-E2}
 \EE_2^\eps(t) =\EE_2\Big(\rho^{\epsilon}_+,\rho^{\epsilon}_-,\alpha^{\epsilon}_+,u^{\epsilon}| \rho_+,\rho_-,\alpha_+,u\Big) :=\sum_{i=+,-} \int_\Omega \Big(R_{i}^\varepsilon \ln \frac{R_{i}^\varepsilon}{R_{i}} - R_i^\eps+R_i\Big)(t,x)\dd x .
 \end{equation}
 \end{itemize}
\begin{remark}
Observe that the energy functional $E^\eps$ is the natural energy associated with the system. A detailed formal derivation of the corresponding energy identity for smooth solutions can be found in \cite{lebot2025low}, see also Section~\ref{conservation-energy}. The quantity $\mathcal{E}_1^\eps$ is also standard (see \cite{lions1996mathematical}). The functional $\EE_2^\eps$ is less classical, although similar quantities appear in the work of P.-L. Lions on compressible models, where the strict convexity of $x\ln x - x + 1$ is used in compactness arguments.
\end{remark}

We can now state the main result of this paper, where we denote the fractional masses of the limit system by $R_{\pm} = \alpha_{\pm} \rho_{\pm}$.

\begin{theorem}\label{mainth}
 [Singular limit of the compressible two-phase system] Let $C_0>0$, $\gamma_\pm\geq 2$ be three real numbers and $\alpha_{\pm,0}$ be two functions such that $\gamma_+\neq \gamma_-$ and $0<\underline{a}\leq \alpha_{\pm,0} \leq \overline{a} <1$. Let $(\rho_+,\rho_-,\alpha_+,u)$ be a solution of the non-homogeneous incompressible Navier-Stokes system~\eqref{limitsys}, verifying the regularity properties of the limit solution, see Theorem~\ref{sol-lim}, and let $T>0$ be the existence time of this solution.
 
For each $\varepsilon>0$, let $(\rho_{+,0}^\varepsilon, R_{+,0}^\varepsilon, m_0^\varepsilon)\in L^{\gamma_+}(\Omega)\times L^{\gamma_+}(\Omega)\times L^1(\Omega)$ be compatible initial data for the bifluid compressible system in the sense of Definition~\ref{rem.compa}, and let $(\rho^{\epsilon}_+,\rho^{\epsilon}_-,\alpha^{\epsilon}_+,u^{\epsilon})$ be a global weak solution of the compressible two-phase Navier--Stokes system~~\eqref{system_eps}, in the sense of Definition~\ref{defp}.

If we assume that the initial data are well prepared in the sense
\[
\EE_1^\eps(0) + \EE_2^\eps(0)\to 0\quad\text{as}\quad \eps\to 0\quad \text{and} \quad \Big(\int_\Omega R_{\pm,0}^\eps\Big)_\eps \text{ bounded,}
\]
then the following holds true when $\eps\to 0$
 \begin{align*}
\rho_i^\varepsilon &\longrightarrow \rho_i &&\text{strongly in } L^\infty(0,T; L^p(\Omega)), \\
\alpha^\varepsilon_i &\longrightarrow \alpha_i &&\text{strongly in } L^q(0,T; L^2(\Omega)),\\
R_\pm^\varepsilon &\longrightarrow R_\pm &&\text{strongly in } L^q(0,T;L^2(\Omega)),\\
u^\varepsilon &\longrightarrow u &&\text{strongly in } L^2(0,T; H^{1}_0(\Omega)),\\
\rho^\eps|u^\eps-u|^2 &\longrightarrow 0 &&\text{strongly in } L^q(0,T; L^1(\Omega)),
\end{align*}
 for all $i=+,-$, all $p< \gamma_i$ and all $q\in [1,\infty)$.
\end{theorem}

\begin{remark}
It will be shown subsequently that the assumption $\int R_{\pm,0}^\eps,\EE_1^\eps(0)$ bounded and $\EE_{2}^\eps(0) \to 0$ yields the convergences $\| R_{\pm,0}^\eps - R_{\pm,0} \|_{L^2(\Omega)} \to 0$ and $\| \alpha_{\pm,0}^\eps - \alpha_{\pm,0} \|_{L^2(\Omega)} \to 0$, see Proposition~\ref{prop.estRL2} and \eqref{2ndterm}.
\end{remark}

\begin{remark}
The strong convergences in $L^\infty(0,T; L^p(\Omega))$ may be quantified by explicit estimates in terms of $\eps$ and $\EE_1^\eps(0)$. For the sake of readability, we do not record these bounds in the statement of the theorem, and refer instead to Proposition~\ref{prop.rhoE1} together with the boundedness of $\EE_1^\eps$ proved in Section~\ref{sect4.1}. In the case where $\gamma_i=2$ then the convergence holds for $p\leq \gamma_i$.
\end{remark}

\begin{remark}
It should be possible to extend this analysis to periodic domains or to two-dimensional settings.
\end{remark}

\subsection{Survey of previous results}\label{survey}

Compressible two-fluid models have been extensively studied in the literature. From a mathematical perspective, we refer for instance to the survey chapter by Bresch, Desjardins, Ghidaglia, Grenier and Hillairet \cite{bresch2018multi} and the references therein. We briefly review below some results that are most relevant to the present work.

In the viscous setting, the analysis of such systems has led to important developments in the theory of compressible PDEs. In particular, the existence of finite-energy weak solutions was established by Bresch, Mucha, and Zatorska \cite{bresch2019finite} as well as by Vasseur, Wen, and Yu \cite{vasseur2019global}, highlighting suitable dissipation mechanisms. In the same spirit, Novotný and Pokorný \cite{novotny2020weak} and Novotný \cite{novotny2020weak1} developed a weak solution theory of compressible bi-fluid models, while Jin, Kwon, Nečasová, and Novotný \cite{jin2021existence} introduced dissipative turbulent solutions and investigated their stability. Weak-strong uniqueness properties were further analyzed in \cite{jin2019weak,li2026weak,li2021remarks}. 
These contributions, among others, as e.g. \cite{li2025non,li2020large,piasecki2022maximal}, form a coherent and evolving theory for compressible two-fluid flows.

For single-phase flows, the low Mach number limit is by now well understood. For the compressible Euler and Navier–Stokes equations, convergence toward incompressible dynamics under well-prepared initial data was first established by Klainerman and Majda \cite{klainerman1981singular, klainerman1982compressible}. This analysis was later extended by Schochet \cite{schochet1994fast}, who developed a general framework to handle fast oscillations in hyperbolic systems.

In the isentropic setting, numerous results are available. In the framework of weak solutions, the low Mach number limit for the compressible Navier–Stokes system was established in \cite{lions1998incompressible}, \cite{desjardins1999incompressible}, \cite{desjardins1999low}, and \cite{bresch2002low}. For strong solutions, we refer to \cite{grenier1997oscillatory} and \cite{danchin2002zero}.

In the non-isentropic case, the analysis is more involved due to the coupling with the temperature (or energy) equation. The low Mach number limit has been justified in the weak solution framework in \cite{feireisl2007low} and \cite{feireisl2013inviscid} (see also the books \cite{lions1998mathematical} and \cite{feireisl2004dynamics}). For strong solutions, results were obtained for instance in \cite{metivier2001incompressible}, \cite{alazard2006low}, and \cite{alazard2005incompressible}.

A concise overview of these developments can be found in \cite{gallagher2005resultats}.

Let us also mention that the incompressible limit for the classical isentropic compressible Navier–Stokes system may be obtained in other asymptotic regimes. In particular, the large viscosity limit leading to the incompressible equations was studied by Danchin and Mucha \cite{danchin2013incompressible}. We may also refer to Fanelli et al~\cite{fanelli2025incompressible} for related incompressible limits in a reduced compressible MHD system.

A comprehensive treatment of the compressible Navier–Stokes system with constant viscosities, with particular emphasis on the relative entropy method as a stability and convergence tool, is given in \cite{feireisl2009singular}. For more complex systems, such as compressible Navier–Stokes equations with density-dependent viscosities, the notion of relative entropy had to be revisited and adapted; see, for instance, \cite{bresch2017relative} and \cite{bresch2019navier}. We also mention works developing the relative entropy framework in related contexts, such as \cite{feireisl2012relative} and \cite{sueur2014inviscid}.

In contrast, significantly fewer results are available for multiphase flow models, even in simplified configurations. Only recently has attention been devoted to the formal derivation of low Mach number limits for various two-phase systems, including models with algebraic or PDE closures, as well as formulations involving one or two velocities. In these works, asymptotic expansions are typically performed as the Mach number tends to zero; see, for instance, \cite{varsakelis2011low} and \cite{lebot2025low}. Numerical investigations of such regimes have also been carried out, see \cite{battisti2025linearly} and the references therein.\newline

In the same week that we posted the first version of this paper on ArXiv, an independent work~\cite{li2026weak} also appeared, dealing with the bi-fluid compressible Navier--Stokes equations. In that paper, the authors investigate weak--strong uniqueness through the introduction of a different relative entropy. Although their functional contains the same natural term $\mathcal{E}_1^\varepsilon$, the additional contribution is simply the $L^2$-distance between the volume fractions $\alpha_i^\varepsilon - \alpha_i$. This choice leads to an entropy of a completely different nature from our functional $\mathcal{E}_2^\varepsilon$ defined in~\eqref{def-E2}, which is built solely on the partial densities $R_i^\varepsilon$. Let us emphasize that the partial densities are the natural variables for the bi-fluid compressible Navier--Stokes system with algebraic closure~\eqref{system_eps}, as highlighted in the majority of the literature (see, among others, \cite{bresch2019finite,novotny2020weak}). 

The month after our initial submission, another related paper appeared on ArXiv: in~\cite{li2026lowmachnumberlimit}, the authors reused the entropy introduced in~\cite{li2026weak} to obtain a low-Mach number limit under very well-prepared initial data. Besides assuming that $\rho_i^\varepsilon(0)$ is close to the constant $C_0^{1/\gamma_i}$---a condition encoded by the uniform bound of $\varepsilon^{-2} H_i(\rho_{i,0}^\varepsilon)$, which is also present in our assumptions and is standard in low-Mach analyses---they further require that $\alpha_i^\varepsilon(0)$ is close to the constant $\overline{\alpha_i} = 1/C_0^{\gamma_i}$ in high Sobolev norms. The additional constraint $\overline{\alpha_+} + \overline{\alpha_-} = 1$ imposes a strong restriction on the admissible choices of $C_0$. In this setting, the authors naturally recover in the limit the homogeneous incompressible Navier--Stokes equations, characterized by the constant density $\rho = \overline{\alpha_+} \rho_+ + \overline{\alpha_-} \rho_- = 2$.
Although studying this particular regime---where the volume fractions remain close to constants---is an interesting and meaningful first step (and is included in our approach), the primary purpose of our entropy functionals is precisely to address genuinely non-trivial mixtures. Our framework handles the full complexity of the bi-fluid model, in which the coupling between $\alpha_i$ and $R_i$ is highly non-linear and non-trivial.

For completeness, we should also emphasize that the analysis in~\cite{li2026lowmachnumberlimit} is entirely independent of ours. Instead of working with global weak solutions of the compressible system, the authors consider local strong solutions on a time interval $T_\varepsilon$ that is smaller than $T$, the existence time of the solution to the limit system. Their analysis should be interpreted as a perturbative analysis around constant states $(\rho_i, \overline{\alpha_i})$, allowing in particular $\gamma_\pm>1$.

\subsection{Strategy of the proof}

The proof is based on the relative entropy method. More precisely, we consider the modulated quantities introduced above, which measure the deviation between a solution of the compressible system and a strong solution of the limit system, and combine them into a relative entropy functional. The main argument then consists in deriving a stability inequality for this functional and applying a Gronwall-type argument to conclude the convergence as $\eps \to 0$.

The strategy is inspired by the approach of Lions and Masmoudi \cite{lions1998incompressible}. However, the two-phase structure of the system introduces additional difficulties. In particular, the standard relative entropy does not provide sufficient control to close the argument, as it fails to yield bounds on the $L^2$-distance between $R_\pm^\eps$ and $R_\pm$. To overcome this issue, we introduce a logarithmic relative entropy, which provides the required control on the densities and their deviations from the limit state. Combined with the modulated entropy, which encodes the natural stability structure of the system, this yields a suitable framework for the analysis. 

With these tools in hand, we proceed in three main steps.

\textit{Step 1: Uniform bounds via a first Gronwall argument}. We first establish uniform bounds for $\EE_1^\eps$ and $\EE_{2}^\eps$. The key observation is that, while each functional alone does not yield sufficient control, their sum does. By deriving a differential inequality for $\EE_1^\eps + \EE_{2}^\eps$ and applying a first Gronwall argument, we obtain uniform bounds on both $\EE_1^\eps$ and $\EE_{2}^\eps$ individually.

\textit{Step 2: Strong convergence of densities and weak convergence of the divergence}. The uniform bound on $\EE_1^\eps$ yields the strong convergence of the densities $\rho_\pm^\eps$. In addition, combining the control of $\EE_1^\eps + \EE_{2}^\eps$ with compactness arguments allows us to deduce the weak convergence of $\div u^\eps$ to $0$.

\textit{Step 3: Strong convergence of the velocity and conclusion}. Finally, a second Gronwall argument ensures the strong convergence of the velocity field $u^\eps$ as well as that of the partial densities. Next, the convergence of the volume fractions relies crucially on the specific structure of the bi-fluid system. 

This closes the argument, as the convergence of all relevant quantities is then established, and the limit system is identified.

\bigskip

This paper is organized as follows. Section~\ref{recall} is devoted to the mathematical setting of the problem and recalls several preliminary results. In Section~\ref{relative_entropy}, we establish some inequalities for the relative entropies.
Next, in Section~\ref{convR}, we find some useful uniform bounds for the relative entropies using a first Gronwall estimate. This allows us to establish the strong convergence of the densities and the weak divergence of $\div u^\eps$. Finally, Section~\ref{End} completes the proof of our main theorem by proving the strong convergences of the velocity field and the volume fractions. Appendix~\ref{app1} provides entropy computations for more general velocity fields $u$, which may be useful for analyzing stability or establishing weak–strong uniqueness.

\section{Preliminaries and mathematical setting}\label{recall}

\subsection{Bi-fluid compressible system and nondimensionalization}\label{nondim}

In the present paper, we analyse a bi-fluid isentropic compressible Navier-Stokes system for barotropic pressures. In this subsection, we present this system and its nondimensionalization. Denoting the phases by $+$ and $-$, we are considering the following system.
\begin{equation}
 \left\{
 \begin{aligned}\label{system}
 &\partial_t (\alpha_+ \rho_+) + \div(\alpha_+\rho_+u) = 0,\\
 &\partial_t (\alpha_- \rho_-) + \div(\alpha_-\rho_-u) = 0,\\
 &\partial_t ((\alpha_+ \rho_+ + \alpha_- \rho_-)u) + \div((\alpha_+\rho_+ + \alpha_-\rho_-)u\otimes u) +\nabla (\alpha_+p_+(\rho_+)+\alpha_-p_-(\rho_-)) \\
 &\quad = \mu \Delta u + (\mu + \lambda)\nabla \div u,\\
 &p_i(\rho_i) = (\rho_i)^{\gamma_i}, \qquad \rho_i \geq 0, \qquad i=+,-,\\
 &p_+(\rho_+) = p_-(\rho_-), \\
 &\alpha_+ + \alpha_- = 1,\qquad 0\leq \ap \leq 1,
 \end{aligned}
 \right .
\end{equation}
where $\rho_i$, $p_i$, $\alpha_i$ the density, pressure and volume fraction of phase $i$. 
Let us notice that, thanks to the algebraic closure, the pressure term in the momentum equation can be simplified:
\[
\nabla (\alpha_+p_+(\rho_+)+\alpha_-p_-(\rho_-)) = \nabla (\alpha_+p_+(\rho_+)+\alpha_-p_+(\rho_+))=\nabla p_+(\rho_+).
\]

In this system, we restrict ourselves to the case of a common velocity field, denoted by $u$. Physically, this corresponds to a regime of rapid mechanical relaxation between the phases, where pressure equilibrium is quickly achieved. Mathematically, such models can be rigorously derived as asymptotic limits of more general two-velocities, two-pressures systems, using homogenization (see, e.g., Bresch et al. \cite{bresch2023mathematical}) or pressure-relaxation arguments (see \cite{burtea2023pressure}, \cite{burtea2026relaxation}). This assumption considerably simplifies the structure of the equations and is standard in the literature, see for instance \cite{novotny2020weak}. Allowing for two distinct velocities would lead to a substantially more intricate coupling mechanism and raises major analytical difficulties; the rigorous study of such models remains a challenging and largely open problem.

The system is endowed with the Dirichlet boundary condition \eqref{eq.bc} for $u(t,x)$ on $(0,T) \times \partial \Omega$.

The incompressible limit is the limit where the speed of the fluid becomes negligible compared to the speed of sound. Thus we reset the flow speed to the scale by replacing $u$ with $\eps u^\eps$. A particle will move a distance of order $1/\eps$ during a time of order $1/\eps^2$, meaning that the relevant time scale is $1/\eps^2$. This is why we can use a change of variable like $t^\eps=\eps^2 t$ and $x^\eps = \eps x$, with $t^\eps$ and $x^\eps$ of order $1$ (see \cite{danchin2005low}). Thus, we scale $u$ and $\rho$ in the following way
\begin{equation*}
 u(t,x) = \eps u^\eps(t^\eps,x^\eps), \quad \rho(t,x) = \rho^\eps(t^\eps,x^\eps).
\end{equation*}
Observe that, under this change of variables, the viscosity coefficients do not require rescaling.

Therefore, System~\eqref{system} can be rewritten as System~\eqref{system_eps}.

\subsection{Definition of a weak solution of the compressible bi-fluid system}

We consider weak solutions to System~\eqref{system_eps} in the following sense.

In this definition, we always set $\alpha_-^\eps = 1-\alpha_+^\eps$.

\begin{definition}\label{defp}
Let us consider some initial data $(\rho_{+,0}^\eps,R_{+,0}^\eps, m_0^\epsilon)\in L^{\gamma_+}(\Omega)\times L^{\gamma_+}(\Omega)\times L^{1}(\Omega)$ for $\gamma_\pm \geq 2$ satisfying the compatibility conditions \eqref{compatibilityeps}, where we defined $\alpha_{\pm,0}^\eps\in [0,1]$, $\rho_{-,0}^\eps\geq 0$, $\rho^\eps_0\geq 0$ and $u_0^\eps$ from $(\rho_{+,0}^\eps,R_{+,0}^\eps, m_0^\epsilon)$ (see Remark~\ref{rem.compa}). Let $T>0$. We say that $(\rho_\pm^\eps,\alpha_\pm^\eps,u^\eps)$ is a finite energy weak solution of the compressible two-phase Navier-Stokes System~\eqref{system_eps} with algebraic closure on $(0,T)$ associated with the initial data $(\rho_{\pm,0}^\eps, \alpha_{\pm,0}^\eps, u_0^\eps)$ if we have
 \begin{itemize}
 \item the regularity:
 \begin{align*}
 &0 \leq \alpha_+^\eps \leq 1, \quad \ap^\eps + \am^\eps = 1, \quad \rho_\pm^\eps \geq 0, \quad \rho_+^{\gamma_+}=\rho_-^{\gamma_-} \text{ a.e. in } (0,T)\times \Omega,\\
 &\rho_\pm^\eps ,\ \alpha_\pm^\eps \rho_\pm^\eps\in C_{\text{weak}}([0,T]; L^{\gamma_\pm}(\Omega))\cap L^\infty(0,T; L^{\gamma_\pm}(\Omega)),\\
 &\sqrt{\rho^\eps} u^\eps \in L^\infty(0,T; L^{2}(\Omega)), \quad
 \rho^\eps u^\eps \in C_{\text{weak}}([0,T]; L^{\frac{2 \gamma_+}{\gamma_++1}}+L^{\frac{2 \gamma_-}{\gamma_-+1}}(\Omega)),\\
 & u^\eps\in L^2(0,T;H^1_0(\Omega)),
 \end{align*}
 
 \item the continuity equation:
 \begin{align*}
 \int_\Omega (\alpha_i^\eps \rho_i^\eps \phi)(t,x) \dd x -\int_\Omega (\alpha_{i,0}^\eps \rho_{i,0}^\eps)(x) \phi(0,x) \dd x = \int_0^t\!\!\! \int_\Omega \Big(\alpha_i^\eps \rho_i^\eps(\dt \phi + u^\eps \cdot \nabla \phi)\Big)(s,x) \dd x \dd s,
 \end{align*}
 for a.e. $t\in (0,T)$, for any $\phi \in \mathcal{C}^\infty ([0,T]\times \overline{\Omega})$ and $i=+,-$,
 \item the momentum equation:
 \begin{align*}
 \int_\Omega (\rho^\eps u^\eps)(t,x)\cdot \phi(t,x) \dd x &- \int_\Omega (\rho^\eps_0 u^\eps_0(x) \cdot \phi(0,x) \dd x \\
 =& \int_0^t\!\!\! \int_\Omega \Big( \rho^\eps u^\eps \cdot \dt \phi + \rho^\eps (u^\eps \otimes u^\eps) : \nabla \phi+ \frac{1}{\eps^2} (\rho_+^{\epsilon})^{\gamma_+} \div \phi \Big)(s,x) \dd x \dd s \\
 & -\int_0^t\!\!\! \int_\Omega \left(\mu \nabla u^\eps : \nabla \phi + (\mu + \lambda) \div u^\eps \div \phi \right)(s,x) \dd x \dd s 
 \end{align*}
 for a.e. $t\in (0,T)$, for any $\phi \in \mathcal{C}_c^\infty ([0,T]\times \Omega;\R^3)$,
 \item the energy inequality:
 \begin{equation}\label{energy_estimate}
 E^\eps(t) +\mu \int_0^t\!\!\! \int_{\Omega}\left( \nabla u^\eps \right)^2 \dd x \dd s +\left(\lambda + \mu\right) \int_0^t\!\!\! \int_{\Omega} (\div u^\eps )^2 \dd x \dd s \leqslant E^\eps(0)
 \end{equation}
 for a.e. $t \in(0, T)$.
 \end{itemize}
\end{definition}

\begin{remark}
The above inequality follows from the classical energy inequality (see, e.g., \cite{novotny2020weak}) and the special structure of the functional $H$, see the next section.
\end{remark}

\subsection{Energy conservation}\label{conservation-energy}

In this section, we provide some details of the formal computation leading to the energy conservation for completeness. We first refer to \cite{lebot2025low} for the terms whose contribution to the energy balance has already been derived
\begin{equation*}
 \begin{aligned}
 &\frac{1}{2} \int_{\Omega} \Big(\rho^\varepsilon |u^\eps|^2\Big)(t,x) \dd x + \int_{\Omega} \frac{\alpha_+^{\epsilon}(t,x)}{\epsilon^2} \frac{\rho_+^{\eps \gp}}{\gp-1}(t,x)\dd x + \int_{\Omega} \frac{\alpha_-^{\epsilon}(t,x)}{\epsilon^2} \frac{\rho_-^{\eps \gm}}{\gm-1}(t,x)\dd x \\
 &\quad \int_0^t \int_\Omega \mu | \nabla u |^2(t,x) + (\mu+\lambda) |\div u |^2 (t,x)\dd x\\
 &\quad \leq \frac{1}{2} \int_{\Omega} \Big(\rho^\varepsilon |u^\eps|^2\Big)(0,x) \dd x + \int_{\Omega} \frac{\alpha_+^{\epsilon}(0,x)}{\epsilon^2} \frac{\rho_+^{\eps \gp}}{\gp-1}(0,x)\dd x + \int_{\Omega} \frac{\alpha_-^{\epsilon}(0,x)}{\epsilon^2} \frac{\rho_-^{\eps \gm}}{\gm-1}(0,x)\dd x.
 \end{aligned}
 \end{equation*}

We now focus on the remaining terms and show that they do not contribute to the energy variation:
\begin{align*}
 &-\int_{\Omega} \frac{\alpha_+^{\epsilon}}{\epsilon^2(\gp-1)} \Big(\rp^{\gp}+\gp \rp^{\gp-1}(\rp^\eps - \rp)\Big) \dd x - \int_{\Omega} \frac{\alpha_-^{\epsilon}}{\epsilon^2(\gm-1)}\Big(\rmi^{\gm}+\gm \rmi^{\gm-1}(\rmi^\eps-\rmi)\Big) \dd x\\
 &= \frac{\rp^{\gp}}{\epsilon^2} \int_{\Omega} \alpha_+^{\epsilon} \dd x - \frac{\gp \rp^{\gp-1}}{\eps^2(\gp-1)} \int_\Omega \ap^\eps \rp^\eps \dd x + \frac{\rmi^{\gm}}{\epsilon^2} \int_{\Omega} \alpha_-^{\epsilon} \dd x - \frac{\gm \rmi^{\gm-1}}{\eps^2(\gm-1)} \int_\Omega \am^\eps \rmi^\eps \dd x\\
 &=\frac{\rp^{\gp}}{\eps^2}\int_\Omega 1 \dd x - \frac{\gp \rp^{\gp-1}}{\eps^2(\gp-1)} \int_\Omega R_+^\eps \dd x - \frac{\gm \rmi^{\gm-1}}{\eps^2(\gm-1)} \int_\Omega R_-^\eps \dd x,
\end{align*}
where we use the equality of pressures $\rp^{\gp}=\rmi^{\gm}=C_0$ and $\ap^\eps+\am^\eps=1$. We can conclude using the conservation of total mass due to the mass equation in System~\eqref{system_eps}:
\begin{equation}\label{eq.consmass}
\int_\Omega R_\pm^\varepsilon(t,x) \, \dd x = \int_\Omega R_{\pm,0}^\varepsilon(x) \, \dd x \qquad \text{for all } t\in (0,T).
\end{equation}

\subsection{Non-homogeneous incompressible system}\label{sec.incomp}

In the case of nonhomogeneous incompressible flows with constant viscosities, the existence of both weak and strong solutions has been investigated in several works. The existence of weak solutions was first established by Kazhikhov~\cite{kazhikov}. Concerning strong solutions, Ladyzhenskaya and Solonnikov~\cite{ladyzhenskaya1978unique} proved local existence, as well as global existence for small initial data (and global solvability in two space dimensions). In sufficiently smooth domains, Lemoine~\cite{lemoine1997non} obtained local strong regular solutions, which are global for small regular data. Okamoto~\cite{okamoto1984equation} proved existence and uniqueness of local weak solutions and obtained global solutions in three dimensions under smallness assumptions on the initial data. Simon~\cite{simon1990nonhomogeneous} established the existence of global-in-time weak solutions under suitable weak assumptions on the initial data, and proved global strong solvability provided the initial density is not too small (and sufficiently regular). The solvability of nonstationary problems for nonhomogeneous incompressible fluids, as well as convergence in the vanishing viscosity limit, was studied by Itoh~\cite{itoh1999solvability}. Choe and Kim~\cite{jun2003strong} established local existence and uniqueness of strong solutions in three dimensions, including situations where the initial density vanishes on an open subset of the domain. More recently, Danchin~\cite{danchin2012lagrangian} developed a well-posedness theory in critical Besov spaces and proved the existence of global-in-time unique solutions under suitable smallness assumptions on the initial data.

Among the vast existing literature, let us state a result of local existence of strong solutions of non-homogeneous viscous incompressible Navier-Stokes equations, coming from \cite{salvi1991equations}.

\begin{theorem}\label{sol-lim}
 Let $V = \{u \in H^1(\Omega) \text{ such that } \div u =0\text{ and }u\cdot n|_{\partial \Omega}=0\}$. Let the boundary $\partial \Omega$ be of class $C^\infty$.
 Let $(\rho_0, u_0) $ be some initial data such that $u_0 \in H^3(\Omega) \cap V$ in $\Omega$ and $\rho_0 \in C^\infty(\bar{\Omega})$, with $\overline{C} \geq \rho_0 \geq \underline{C}>0$ for some constants $\overline{C}, \underline{C}$.
 Then there exists $T>0$ and a unique solution $(\rho,u, \Pi)$ of System~\eqref{limitsys} that satisfy
 \begin{align*}
 &u \in L^\infty(0,T;V) \cap C(0,T;H^3(\Omega)\cap V), \qquad \dt u \in C(0,T;H^1(\Omega)\cap V) \cap L^2([0,T]\times \Omega),\\
 &\rho \in C^1([0,T]\times \Omega), \qquad \dt \rho \in L^\infty(0,T;H^{-1}(\Omega)),\\
 &\nabla \Pi \in L^\infty(0,T;L^3(\Omega))\cap L^2(0,T;H^1(\Omega)).
 \end{align*}
\end{theorem}

\begin{remark}\label{rkalpha}
When we consider $(\rho, u)$ a regular solution to non-homogeneous viscous incompressible Navier-Stokes equations, it is clear that we may construct a solution of \eqref{limitsys}. Indeed, solving $\rho=\rp \ap + \rmi \am$ and $1= \ap + \am$, we simply define
\[
\ap = \frac{\rho - \rmi}{\rp-\rmi} \quad\text{and}\quad \am = -\frac{\rho - \rp}{\rp-\rmi}
\]
which is well defined for $\gamma_+\neq \gamma_-$, as $C_0>0$. We note that $\alpha_\pm$ satisfy a transport equation. If the initial data verify $\alpha_{\pm,0} \in [\underline{a}, \overline{a}] \subset (0,1)$, then this property is preserved for all times. Likewise, since $\rho$ also satisfies a transport equation, it follows that $0<\underline{C} \leq \rho \leq \overline{C}$. Moreover, the constants $\underline{C}$ and $\overline{C}$ can be expressed in terms of $\underline{a}$, $\overline{a}$, $C_0$, and $\gamma_\pm$.
\end{remark}

\begin{remark}
The only difference compared to Theorem 3 of Salvi is that we additionally assume that $\nabla \Pi$ has the same regularity as $\Delta u$. We do not claim that this regularity is optimal for our analysis; however, we rely at some point on the fact that $\Delta u$ and $\nabla \Pi$ belong to $L^\infty(0,T;L^3(\Omega))\cap L^2(0,T;H^1(\Omega))$, as guaranteed by Theorem~\ref{sol-lim}.
\end{remark}

\section{Inequalities for the relative entropies}\label{relative_entropy}

In this section, we establish some inequalities on the relative entropies defined in \eqref{def-Energy}-\eqref{def-E2} which will allow us to prove Theorem~\ref{mainth}.

In the following, $C$ denotes a generic constant which may depend, depending on the context, on $\underline{C},\overline{C},\underline{a},\overline{a},C_0$ or $\gamma_\pm$.

\subsection{First part of the entropy}\label{computationsE1}

We begin by estimating the time derivative of $\EE^\eps_1$ for sufficiently smooth functions $(\rho_+,\rho_-,\alpha_+,u)$ without assuming that they solve System~\eqref{system_eps}.

\begin{proposition}\label{prop1} If $(\rho^{\epsilon}_+,\rho^{\epsilon}_-,\alpha^{\epsilon}_+,u^{\epsilon})$ is a weak solution of System~\eqref{system_eps}, defined as in Definition~\ref{defp}, then for all $(\rho_+,\rho_-,\alpha_+,u)$ functions verifying the regularity properties of the limit solution, see Theorem~\ref{sol-lim}, such that $p_+(\rho_+) =p_-(\rho_-) = C_0>0$ and $\div u=0$, we have
 \begin{align*}
 \EE_1^\eps(t) - \EE_1^\eps(0) \leq& - \mu \int_0^t\!\!\! \int_{\Omega} \left( \nabla (u^{\epsilon}-u) \right)^2 - (\lambda+\mu) \int_0^t\!\!\! \int_{\Omega} \left( \div (u^{\epsilon}-u) \right)^2 \dd x \dd s \\
 &-\mu \int_0^t\!\!\! \int_{\Omega} \nabla u : \nabla(u^\eps-u) \dd x\dd s \\
 &+ \int_0^t\!\!\! \int_{\Omega} \rho^\eps (\partial_t u + u \cdot \nabla u)\cdot (u-u^{\epsilon}) \dd x \dd s \\
 &+ \int_0^t\!\!\! \int_{\Omega} \rho^\eps (u^{\epsilon}-u)\cdot \nabla u \cdot (u-u^{\epsilon}) \dd x \dd s. 
 \end{align*}
\end{proposition}

\begin{proof}
We have
\begin{align*}
 \EE_1^\eps(t) - \EE_1^\eps(0) =& E^\eps(t) - E^\eps(0) \\
 &+\frac{1}{2} \int_{\Omega} \Big(\rho^\eps |u|^2\Big)(t,\cdot) - \frac{1}{2} \int_{\Omega} \Big(\rho^\eps |u|^2\Big)(0,\cdot) \\
 &- \int_{\Omega} \Big(\rho^\eps u^\eps \cdot u\Big)(t,\cdot) + \int_{\Omega} \Big(\rho^\eps u^\eps \cdot u\Big)(0,\cdot) .
\end{align*}

Let us consider mass equations with $\phi=\frac{|u|^2}{2}$ and add the two, we get:
\begin{align*}
 \frac{1}{2} \int_{\Omega} \Big(\rho^\eps |u|^2\Big)(t,\cdot) - \frac{1}{2} \int_{\Omega} \Big(\rho^\eps |u|^2\Big)(0,\cdot) 
 &=\frac{1}{2} \int_0^t\!\!\!\int_{\Omega} \rho^\eps \left(\partial_t|u|^2+u^\eps \cdot \nabla |u|^2\right) \nonumber \\
&= \int_0^t \!\!\!\int_\Omega \rho^\eps u\cdot \Big(\partial_t u + (u^{\epsilon} \cdot \nabla) u\Big).
\end{align*}
Let us notice that we need, here and below, to check that all terms are well defined by using the regularity properties of the functions depending on $\eps$ (see Definition~\ref{defp}) and the smooth functions (see Theorem~\ref{sol-lim}).

Let us now use the momentum equation by taking $\phi=u$, we get:
\begin{multline*}
 -\int_{\Omega} \Big(\rho^\eps u^\eps \cdot u\Big)(t,\cdot) + \int_{\Omega} \Big(\rho^\eps u^\eps \cdot u\Big)(0,\cdot) = -\int_0^t\!\!\!\int_{\Omega}\rho^\eps u^{\varepsilon} \cdot (\partial_t u + (u^{\varepsilon} \cdot \nabla) u )
 \\
 - \frac{1}{\epsilon^2}\int_0^t\!\!\! \int_{\Omega} \rho_{+}^{\varepsilon \gamma_+}\div u 
 +\mu \int_0^t\!\!\!\int_{\Omega} \nabla u^{\varepsilon} : \nabla u + (\mu+\lambda) \int_0^t\!\!\!\int_{\Omega} \left(\operatorname{div} u^{\varepsilon}\right)(\operatorname{div} u).
\end{multline*}

We can rewrite 
\begin{equation*}
\begin{aligned}
 \int_{\Omega} \rho^\eps (\partial_t u + (u^{\epsilon} \cdot \nabla) u)\cdot (u-u^{\epsilon})
 =& \int_{\Omega} \rho^\eps (\partial_t u + (u \cdot \nabla) u)\cdot (u-u^{\epsilon}) \\
 &+ \int_{\Omega}\rho^\eps \Big(((u^{\epsilon}-u)\cdot \nabla) u \Big)\cdot (u-u^{\epsilon}).
\end{aligned}
\end{equation*}

Therefore, using the energy estimate \eqref{energy_estimate} satisfied by $(\rho_+^{\epsilon}, \rho_-^{\epsilon}, \alpha_+^{\epsilon},u^{\epsilon}) $, we obtain:
\begin{align*}
 \EE_1^\eps(t) - \EE_1^\eps(0) \leq& -\mu \int_0^t \!\!\!\int_{\Omega}\left( \nabla u^\eps \right)^2 \dd x \dd t - \left(\lambda + \mu\right) \int_0^t \!\!\!\int_{\Omega} (\div u^\eps )^2 \dd x \dd t\\
 &+\int_0^t\!\!\! \int_{\Omega} \rho^\eps (\partial_t u + u \cdot \nabla u)\cdot (u-u^{\epsilon}) \\
 &+ \int_0^t \!\!\! \int_{\Omega} \rho^\eps (u^{\epsilon}-u)\cdot \nabla u \cdot (u-u^{\epsilon})\\
 & - \frac{1}{\epsilon^2}\int_0^t\!\!\! \int_{\Omega} \rho_{+}^{\varepsilon \gamma_+}\div u \\
 &+\mu \int_0^t \!\!\! \int_{\Omega} \nabla u^{\varepsilon} : \nabla u + (\mu+\lambda) \int_0^t \!\!\! \int_{\Omega} \left(\operatorname{div} u^{\varepsilon}\right)(\operatorname{div} u).
\end{align*}

By virtue of monotonicity of the operators $\nabla$ and $\div$, namely, we write
\begin{align*}
 - \nabla u^{\varepsilon} : \nabla (u^\eps-u) &= - \left(\nabla (u^\eps - u) \right)^2 - \nabla u : \nabla(u^\eps - u),
\end{align*}
and
\begin{align*}
 - \div u^\eps \div(u^\eps-u) &= - \left( \div(u^\eps - u) \right)^2 - \div u \div(u^\eps - u),
\end{align*}
which ends the proof because $u$ is divergence free.
\end{proof}

\begin{remark}
Appendix~\ref{app1} presents entropy computations for more general velocity fields $u$, which extend those considered here and could be used to study stability or weak-strong uniqueness.
\end{remark}

We now consider $(\rho_+,\rho_-,\alpha_+,u)$ as a solution of System~\eqref{limitsys}, verifying the regularity properties of the limit solution, see Theorem~\ref{sol-lim}. We can reformulate Proposition~\ref{prop1}.

\begin{proposition}\label{prop2}
If $(\rho^{\epsilon}_+,\rho^{\epsilon}_-,\alpha^{\epsilon}_+,u^{\epsilon})$ is a weak solution of System~\eqref{system_eps}, defined as in Definition~\ref{defp}, and if $(\rho_+,\rho_-,\alpha_+,u)$ is a strong solution of System~\eqref{limitsys}, then
\begin{align*}
\EE_1^\eps(t) - \EE_1^\eps(0) \leq& - \mu \int_0^t\!\!\! \int_{\Omega} \left( \nabla (u^{\epsilon}-u) \right)^2 - (\lambda+\mu) \int_0^t\!\!\! \int_{\Omega} \left( \div (u^{\epsilon}-u) \right)^2 \dd x \dd s \\
 &+ \mu \int_0^t\!\!\! \int_\Omega \frac{\rho^\eps - \rho}{\rho} \Delta u \cdot(u-u^\eps) 
 + \int_0^t\!\!\! \int_{\Omega} \rho^\eps (u-u^\eps) \cdot \nabla u \cdot (u^\eps-u) \\
 &+ \int_0^t\!\!\! \int_{\Omega} \frac{\rho-\rho^\eps}{\rho} \nabla \Pi\cdot (u-u^\eps)- \int_0^t\!\!\! \int_{\Omega} \nabla \Pi\cdot (u-u^\eps). 
\end{align*}
\end{proposition}

\begin{proof}
Let us define the remainder $\RR$ as almost all terms on the right-hand side of Proposition~\ref{prop1}:
 \begin{align*}
 \RR^\eps := &-\mu \int_0^t\!\!\! \int_{\Omega} \nabla u : \nabla(u^\eps-u) \\
 &+ \int_0^t\!\!\! \int_{\Omega} \rho^\eps (\partial_t u + u \cdot \nabla u)\cdot (u-u^{\epsilon}) \\
 &+ \int_0^t\!\!\! \int_{\Omega} \rho^\eps (u^{\epsilon}-u)\cdot \nabla u \cdot (u-u^{\epsilon}).
 \end{align*}

We use the momentum equation verified by $(\rho_+,\rho_-,\alpha_+,u)$ to get:
\begin{multline*}
 \RR^\eps = \int_0^t\!\!\! \int_{\Omega} \frac{\rho^\eps}{\rho} \mu \Delta u (u-u^{\epsilon}) +\mu \int_0^t\!\!\! \int_{\Omega} \nabla u : \nabla(u-u^{\epsilon})\\
 -\int_0^t\!\!\! \int_{\Omega} \frac{\rho^\eps}{\rho} \nabla \Pi \cdot (u-u^\eps) + \int_0^t\!\!\! \int_{\Omega} \rho^\eps (u-u^\eps) \cdot \nabla u \cdot (u^\eps-u).
\end{multline*}

Finally, we end the proof by noticing that 
\begin{align*}
 \int_{\Omega} \mu \frac{\rho^\eps}{\rho} \Delta u (u-u^{\epsilon}) + \int_{\Omega} \mu \nabla u : \nabla(u-u^\eps) &= \mu \int_\Omega \frac{\rho^\eps}{\rho}\Delta u (u-u^\eps)- \mu \int_\Omega \Delta u (u-u^\eps) \\
 &= \mu \int_\Omega \frac{\rho^\eps - \rho}{\rho} \Delta u (u-u^\eps).
\end{align*}

\end{proof}

\subsection{Logarithmic entropies}

We now study the evolution of the entropy component \eqref{def-E2} for weak solutions of System~\eqref{system_eps}. 
We begin by estimating the time derivative of $\EE^\eps_2$ for sufficiently smooth functions $(\rho_+,\rho_-,\alpha_+,u)$ without assuming that they solve System~\eqref{limitsys}.

\begin{proposition}\label{prop1E2}
If $(\rho^{\epsilon}_+,\rho^{\epsilon}_-,\alpha^{\epsilon}_+,u^{\epsilon})$ is a weak solution of System~\eqref{system_eps}, defined as in Definition~\ref{defp}, then for all $(\rho_+,\rho_-,\alpha_+,u)$ functions verifying the regularity properties of the limit solution, see Theorem~\ref{sol-lim}, we have
\begin{align*}
 \int_\Omega R_i^\eps \ln\Big(\frac{R_i^\eps}{R_i} \Big) (t) - \int_\Omega R_i^\eps \ln\Big(\frac{R_i^\eps}{R_i} \Big)(0) =& \int_0^t\!\!\! \int_\Omega (R_i^\eps - R_i)(u - u^\eps) \nabla \ln R_i \\
 &+ \int_0^t\!\!\! \int_\Omega (R_i^\eps - R_i) \div (u-u^\eps)\\
 &- \int_0^t\!\!\! \int_\Omega \frac{R_i^\eps}{R_i} (\dt R_i + \div ( R_iu)),
\end{align*}
for a.a. $t\in (0,T)$ and $i=+,-$.
\end{proposition}

\begin{proof}
We want to take $\phi = \ln R_\pm^\eps$ in the continuity equation of Definition~\ref{defp}. To justify the use of a possibly non-admissible test function, we appeal to the theory of renormalized solutions. Since 
\[
u^\varepsilon \in L^2(0,T; H^1(\Omega)),
\]
we regularize the continuity equation by convolution.
At the regularized level, the quantities are smooth enough to perform the computations rigorously. This procedures is classical, in particular, it is used by P.-L. Lions in the barotropic compressible Navier--Stokes system to obtain compactness of the density. It requires
\[
R^\eps_\pm \in L^2(0,T; L^2(\Omega)),
\]
if one wishes to avoid the truncation--regularization method of E. Feireisl.

We have for a.a. $t\in (0,T)$ and $i=+,-$
\begin{align*}
 \int_\Omega (R_i^\eps \ln R_i^\eps)(t,\cdot) -\int_\Omega R_{i,0}^\eps \ln R_{i,0}^\eps &= \int_0^t\!\!\! \int_\Omega R_i^\eps (\dt \ln R_i^\eps + u^\eps \cdot \nabla \ln R_i^\eps)\\
 &= \int_0^t\!\!\! \int_\Omega \dt R_i^\eps + u^\eps \cdot \nabla R_i^\eps \\
 &= - \int_0^t\!\!\! \int_\Omega R_i^\eps \div u^\eps 
\end{align*}
where we have used the conservation of the total mass, see \eqref{eq.consmass}.

Next, we use the mass equation of $R_i^\eps$ in the weak sense taking $\phi = \ln R_i$, we obtain
\begin{align*}
 \int_\Omega (R_i^\eps \ln R_i)(t,\cdot) -\int_\Omega R_{i,0}^\eps \ln R_{i,0} 
 &=\int_0^t\!\!\! \int_\Omega R_i^\eps (\dt \ln R_i + u^\eps \cdot \nabla \ln R_i)\\
 &= \int_0^t\!\!\! \int_\Omega \frac{R_i^\eps}{R_i} (\dt R_i + u^\eps \cdot \nabla R_i)\\
 &= \int_0^t\!\!\! \int_\Omega R_i^\eps (u^\eps - u) \nabla \ln R_i + \int_0^t\!\!\! \int_\Omega \frac{R_i^\eps}{R_i} (\dt R_i + u \cdot \nabla R_i).
\end{align*}

Subtracting these two equalities, we get
\begin{align*}
 \int_\Omega \Big(R_i^\eps \ln \frac{R_i^\eps}{R_i}\Big)(t,x) \dd x &-\int_\Omega R_{i,0}^\eps(x) \ln \frac{R_{i,0}^\eps(x)}{R_{i,0}(x)} \dd x\\
 =& \int_0^t\!\!\! \int_\Omega R_i^\eps (u - u^\eps) \nabla \ln R_i - \int_0^t\!\!\! \int_\Omega R_i^\eps \div u^\eps - \int_0^t\!\!\! \int_\Omega \frac{R_i^\eps}{R_i} (\dt R_i + u \cdot \nabla R_i) \\
 =& \int_0^t\!\!\! \int_\Omega (R_i^\eps - R_i)(u - u^\eps) \nabla \ln R_i 
 + \int_0^t\!\!\! \int_\Omega R_i (u-u^\eps) \nabla \ln R_i \\
 &+ \int_0^t\!\!\! \int_\Omega R_i^\eps \div (u-u^\eps)-\int_0^t\!\!\! \int_\Omega R_i^\eps \div u- \int_0^t\!\!\! \int_\Omega \frac{R_i^\eps}{R_i} (\dt R_i + u \cdot \nabla R_i) \\
 =& \int_0^t\!\!\! \int_\Omega (R_i^\eps - R_i)(u - u^\eps) \nabla \ln R_i 
 + \int_0^t\!\!\! \int_\Omega (R_i^\eps - R_i) \div (u-u^\eps)\\
 &- \int_0^t\!\!\! \int_\Omega \frac{R_i^\eps}{R_i} (\dt R_i + \div ( R_iu)).
\end{align*}

\end{proof}

We now consider $(\rho_+,\rho_-,\alpha_+,u)$ as a solution of System~\eqref{limitsys}, verifying the regularity properties of the limit solution, see Theorem~\ref{sol-lim}. We can reformulate Proposition~\ref{prop1E2}.

\begin{proposition}\label{proprelR}
If $(\rho^{\epsilon}_+,\rho^{\epsilon}_-,\alpha^{\epsilon}_+,u^{\epsilon})$ is a weak solution of System~\eqref{system_eps}, defined as in Definition~\ref{defp}, and if $(\rho_+,\rho_-,\alpha_+,u)$ is a solution of System~\eqref{limitsys}, verifying the regularity properties of the limit solution, see Theorem~\ref{sol-lim}, then there exits $C>0$ depending on $\|\nabla R_i\|_{L^3}$ such that
\begin{equation*}
 \EE_2^\eps (t) - \EE_2^\eps(0) \leq 
 C \int_0^t \sum_{i=+,-}\| R_i^\eps - R_i\|_{L^2(\Omega)} \| \nabla (u^\eps -u)\|_{L^2(\Omega)} ,
\end{equation*}
for a.a. $t\in (0,T)$.
\end{proposition}

\begin{proof}
Considering $R_\pm$ a solution of System~\eqref{limitsys}, we can rewrite the expression obtained in Proposition~\ref{prop1E2}
\begin{align*}
 \int_\Omega \Big(R_i^\eps \ln \frac{R_i^\eps}{R_i}\Big)(t,x) \dd x &-\int_\Omega R_{i,0}^\eps(x) \ln \frac{R_{i,0}^\eps(x)}{R_{i,0}(x)} \dd x\\
 =& \int_0^t\!\!\! \int_\Omega (R_i^\eps - R_i)(u - u^\eps) \nabla \ln R_i 
 + \int_0^t\!\!\! \int_\Omega (R_i^\eps - R_i) \div (u-u^\eps).
\end{align*}

Using Hölder inequality for the first term we get 
\begin{align*}
 \int_\Omega \lvert (R_i^\eps - R_i)(u - u^\eps) \nabla \ln R_i \rvert &\leq \| R_i^\eps - R_i \|_{L^2(\Omega)} \lVert \nabla \ln R_i \rVert_{L^3(\Omega)} \lVert u-u^\eps \rVert_{L^6(\Omega)},
\end{align*}
hence, by Sobolev embedding, we obtain
\begin{align*}
 \int_\Omega \lvert (R_i^\eps - R_i)(u - u^\eps) \nabla \ln R_i \rvert &\leq C \| R_i^\eps - R_i \|_{L^2(\Omega)} \lVert \nabla ( u-u^\eps) \rVert_{L^2(\Omega)}^2.
\end{align*}

By Poincaré inequality, we also get for the second term
\[
\int_\Omega \lvert (R_i^\eps - R_i) \div (u^\eps - u) \rvert \leq C \| R_i^\eps - R_i \|_{L^2(\Omega)} \| \div(u-u^\eps)\|_{L^2(\Omega)}.
\]
We reach the desired conclusion by the conservation of the total mass for $\int R_i^\eps$ and $\int R_i$.
\end{proof}

\subsection{Control of the densities and the partial densities by the relative entropies}

We work under the hypotheses of the main theorem~\ref{mainth}, in particular, we will always denote by $(\rho^{\epsilon}_+,\rho^{\epsilon}_-,\alpha^{\epsilon}_+,u^{\epsilon})$ a weak solution of System~\eqref{system_eps} in the sense of Definition~\ref{defp} and $(\rho_+,\rho_-,\alpha_+,u)$ a solution of System~\eqref{limitsys}, verifying the regularity properties of Theorem~\ref{sol-lim}.

We first recall some known inequalities about the function $H$ given by \eqref{H} (see \cite{lions1998incompressible}, \cite{feireisl2012relative}, \cite{sueur2014inviscid}). Assume some $\gamma > 1$. For any compact $K \subset (0,+\infty)$ there exists two constants $c_1$ and $c_2$ such that for any $\rho \geq 0$ and for any $r \in K$,
\begin{align*}
 c_1 \left( \lvert \rho - r \rvert^2 1_{\lvert \rho - r\rvert <1} + \lvert \rho - r \rvert^\gamma 1_{\lvert \rho - r\rvert \geq 1} \right) \leq H(\rho | r),\\
 H(\rho | r) \leq c_2 \left( \lvert \rho - r \rvert^2 1_{\lvert \rho - r\rvert <1} + \lvert \rho - r \rvert^\gamma 1_{\lvert \rho - r\rvert \geq 1} \right).
\end{align*}

Choosing $K=[\min(\rho_-,\rho_+),\max(\rho_-,\rho_+)]$, we deduce from our assumption $\gamma_\pm\geq 2$ that, for $\rho_i=C_0^{1/\gamma_i}$ as in the limit system~\eqref{limitsys}, we have for a.e $(t,x)\in (0,T)\times \Omega$ and $i=+,-$
\begin{equation}\label{eqB}
 \alpha_i^\eps |\rho_i^\eps - \rho_i |^2+\alpha_i^\eps |\rho_i^\eps - \rho_i |^{\gamma_i}\leq C \alpha_i^\eps H_i( \rho_i^\eps \vert \rho_i ),
\end{equation}
where $C$ depends only on $C_0$ and $\gamma_\pm$.

We begin by establishing the following control of the densities by $\EE_1^\eps$.

\begin{proposition}\label{prop.rhoE1}
 If $(\int R_{\pm,0}^\eps)_\eps$ is bounded, then for all $i=+,-$ and $p<\gamma_i$, there exists $\beta_{1,i}(p), \beta_{2,i}(p)>0$ depending only on $p$ and $\gamma_\pm$, and $C_p>0$ independent of $\eps$ such that for all $t\geq 0$
 \[
 \| (\rho_i^\eps - \rho_i)(t,\cdot) \|_{L^p(\Omega)} \leq C_p \Big(\eps^{2} \EE_1^\eps(t) \Big)^{\beta_{1,i}(p)}+C_p\Big(\eps^{2} \EE_1^\eps(t) \Big)^{\beta_{2,i}(p)}.
 \]
\end{proposition}
In the case where $\gamma_i=2$ then the convergence holds for $p\leq \gamma_i$.

\begin{proof}
As all the problem is symmetric in terms of $\gamma_+$, $\gamma_-$, we can assume without loss of generality that $\gamma_+>\gamma_-\geq 2$. Alternatively, we can exchange $\alpha_\pm, \rho_\pm$ by $\alpha_\mp, \rho_\mp$ in this subsection. Hence $\gamma = \gp/\gm >1$. The equality of pressures gives
\[
\rho_-^\eps = \rho_+^{\eps\gamma},
\]
so that we can write
\begin{align*}
\int_\Omega \alpha_+^\varepsilon \lvert \rho_+^\varepsilon - \rho_+ \rvert^2 
+ \int_\Omega \alpha_-^\varepsilon \lvert (\rho_+^\varepsilon)^\gamma - \rho_+^\gamma \rvert^2
=\int_\Omega \alpha_+^\varepsilon \lvert \rho_+^\varepsilon - \rho_+ \rvert^2 
+ \int_\Omega \alpha_-^\varepsilon \lvert \rho_-^\varepsilon - \rho_- \rvert^2 \leq C\eps^2 \EE_1^\eps(t).
\end{align*}
where we have used the inequality \eqref{eqB} and the definition of $\EE_1^\eps$ \eqref{def-E1}.

For all $x \in \Omega$ such that $\rho_+^\eps < \rho_+/2$, we have
$$|\rho_+^{\eps \gamma} -\rho_+^{ \gamma}|=\rho_+^{\gamma} -\rho_+^{\eps
\gamma}\geq \rho_+^{\gamma} - ( \rho_+/2)^\gamma 
\geq (1-2^{-\gamma})\frac{\rho_+^{\gamma}}{\rho_+} | \rho_+^{\eps} -\rho_+| ,$$
because $| \rho_+^{\eps} -\rho_+| = \rho_+ - \rho_+^{\eps} \leq \rho_+$

Otherwise, for all $x$ such that $\rho_+^\eps > \rho_+/2$, we have by the mean value theorem applied to $F(x) = x^{1/\gamma}$
$$|\rho_+^{\eps} -\rho_+| \leq \sup_{x>(\rho_+/2)^{\gamma}} F'(x) | \rho_+^{\eps\gamma} -\rho_+^\gamma| =\frac1\gamma (\rho_+/2) ^{1-\gamma} | \rho_+^{\eps\gamma} -\rho_+^\gamma|.$$

Hence, using $\alpha^\eps_+ +\alpha_-^\eps=1$, we obtain the bound
\begin{equation*}
\| \rho_+^\varepsilon - \rho_+ \|_{L^2(\Omega)}^2 = 
 \int_\Omega \alpha_+^\varepsilon \lvert \rho_+^\varepsilon - \rho_+ \rvert^2 
+ \int_\Omega \alpha_-^\varepsilon \lvert \rho_+^\varepsilon - \rho_+ \rvert^2\leq C\eps^2 \EE_1^\eps(t).
\end{equation*}
for some constant $C>0$. Such an inequality gives the desired result for $i=+$ and $p\leq 2$, where $\beta_1(p)=\beta_2(p)=1/2$.

To get the result for $p<\gamma_+$, we show that $\rp^\eps$ is uniformly bounded in $L^{\gp}(\Omega)$. Indeed, using the definition of $H_\pm$ \eqref{H}, the pressure equality gives
\begin{align*}
 \ap^\eps (\rp^\eps)^{\gp} &= (\gp-1) \ap^\eps H_+(\rp^\eps | \rp) -(\gp-1) \ap^\eps \rp^{\gp} + \gp \rp^{\gp-1} \ap^\eps \rp^\eps,\\
 \am^\eps (\rp^\eps)^{\gp} &= \am^\eps (\rmi^\eps)^{\gm} = (\gm-1) \am^\eps H_-(\rmi^\eps | \rmi) -(\gamma_--1) \am^\eps \rmi^{\gm} + \gm \rmi^{\gm-1} \am^\eps \rmi^\eps.
\end{align*}
Summing these inequalities and using that $\alpha_i^\eps \in [0,1]$, that $\rho_i$ are constants and that $\int R_i^\eps=\int R_{i,0}^\eps$ is bounded, we obtain after integration over $\Omega$
\begin{equation}\label{est.rhoLgamma}
\int_\Omega (\rmi^\eps)^{\gm} = \int_\Omega (\rp^\eps)^{\gp} \leq C \eps^2 \EE_1^\eps(t) + C. 
\end{equation}
Hence,
\[
\|\rp^\eps - \rp\|_{L^{\gamma_+}(\Omega)} \leq (C \eps^2 \EE_1^\eps(t) + C)^{1/\gamma_+} +C \leq C (\eps^2 \EE_1^\eps(t))^{1/\gamma_+} +C
\]
By interpolation, we get for any $p\in (2,\gamma_+)$
\[
\|\rp^\eps - \rp\|_{L^{p}(\Omega)} \leq (C\eps^2 \EE_1^\eps(t))^{a/2} \Big(C (\eps^2 \EE_1^\eps(t))^{1/\gamma_+} +C\Big)^{1-a} \leq C(\eps^2 \EE_1^\eps(t))^{\frac{a}2} + C(\eps^2 \EE_1^\eps(t))^{\frac{a}2+\frac{1-a}{\gamma_+} }
\]
where $\frac{a}2,\frac{a}2+\frac{1-a}{\gamma_+} \in (0,\frac{1}2]$ because $\gamma_+ \geq 2$. This ends the proof for $i=+$.

We are proving now that the same type of estimates holds for $\rho_-^\varepsilon$.
Let us remark that $\gamma = \gamma_+/\gamma_-$ satisfies $2<2\gamma<\gamma_+$ because we are considering the case in this proof $2<\gamma_-<\gamma_+$. By the previous inequality applied to $p=2\gamma$, we have proved that
\[
\int_{\{\rho_+^\varepsilon \ge 2\rho_+\}}
|\rho_+^\varepsilon-\rho_+|^{2\gamma} \leq C(\eps^2 \EE_1^\eps(t))^{\beta_1 2\gamma}+C(\eps^2 \EE_1^\eps(t))^{\beta_2 2\gamma} ,
\]
for some $C>0$ and $\beta_1,\beta_2\in (0,1/2)$.
Since $\frac{x}{x-\rho_+}\le C$ for all $x\ge 2\rho_+$, it follows that
\begin{align*}
 \int_{\{\rp^\eps \geq 2\rp\}} |\rp^\eps|^{2\gamma} \leq C^{2\gamma}\int_{\{\rp^\eps \geq 2\rp\}} |\rp^\eps - \rp |^{2\gamma} \leq C(\eps^2 \EE_1^\eps(t))^{\beta_1 2\gamma}+C(\eps^2 \EE_1^\eps(t))^{\beta_2 2\gamma}.
\end{align*}
Using the equality of pressures $\rho_-^\varepsilon=(\rho_+^\varepsilon)^\gamma$, we estimate
\begin{equation*}
 |\rmi^\eps - \rmi | = |(\rp^\eps)^{\gamma} - \rp^{\gamma}| \leq \left\{
 \begin{array}{lll}
 &\gamma(2\rp)^{\gamma-1} |\rp^\eps - \rp| &\mbox{ if } \rp^\eps \leq 2\rp\\
 & (\rp^\eps)^\gamma &\mbox{ if } \rp^\eps \geq 2 \rp
 \end{array}
\right.
\end{equation*}
Squaring, integrating over $\Omega$ and splitting the integral on $\{ \rp^\eps < 2\rp\}$ and $\{ \rp^\eps \geq 2\rp\}$, yields
\begin{equation*}
 \int_\Omega |\rmi^\eps - \rmi |^2 \leq C(\rp) \int_\Omega |\rp^\eps - \rp|^2 +\int_{\{\rp^\eps \geq 2\rp\}} |\rp^\eps|^{2\gamma}
 \leq C(\eps^2 \EE_1^\eps(t))^{\beta_3}+C(\eps^2 \EE_1^\eps(t))^{\beta_4}
\end{equation*}
for some $\beta_3,\beta_4>0$ depending only on $\gamma_\pm$. This gives the result for $i=-$ and $p\leq 2$. As for $i=+$, we deduce by interpolation and the estimate of $\|\rho_-^\eps\|_{L^{\gm}}$ \eqref{est.rhoLgamma} the result for $i=-$ and $p\in (2,\gamma_-)$ if $\gamma_->2$.
\end{proof}

The next result concerns the control of the $L^2$ norm of the partial densities by $\EE_1^\eps$ and $\EE_2^\eps$.

\begin{proposition}\label{prop.estRL2}
 If $(\int R_{\pm,0}^\eps)_\eps$ is bounded, then for all $i=+,-$ there exists $\beta_{1,i}, \beta_{2,i}>0$ depending only on $\gamma_\pm$, and $C>0$ independent of $\eps$ such that for all $t\geq 0$
 \[
 \| (R_i^\eps - R_i)(t,\cdot) \|_{L^2(\Omega)}^2 \leq C \Big(\eps^{2} \EE_1^\eps(t) \Big)^{\beta_{1,i}}+C\Big(\eps^{2} \EE_1^\eps(t) \Big)^{\beta_{2,i}} + C \EE_2^\eps (t).
 \]
\end{proposition}
\begin{proof}
Let $\tilde r>0$ a number that will be fixed later on. Using that the function
\[
\frac{r\ln r -(r-1)}{(r-1)^2}
\]
is continuous on $[0,\tilde r]$ and positive by the convexity of $r\mapsto r\ln r$, there exists $C(\tilde r)>0$ such that
\[
(r-1)^2 \leq C(\tilde r)\Big( r\ln r -(r-1)\Big) , \quad \forall r\in [0,\tilde r].
\]
Writing this inequality for $r=R_i^\eps/R_i$, we have
\[
(R_i^\eps-R_i)^2 \leq C(\tilde r)R_i \Big( R_i^\eps\ln \frac{R_i^\eps}{R_i} -R_i^\eps+R_i\Big) , \quad \text{on } \{ R_i^\eps \leq \tilde r R_i \}
\]
Splitting the integral in two domains, we have
\[
\| R_i^\eps -R_i \|_{L^2(\Omega)}^2 \leq C(\tilde r)\EE_2^\eps(t) + C\int_{R_i^\eps \geq \tilde r R_i } (1 + \rho_i^{\eps 2} )
\]

When $R_i^\eps \geq \tilde r R_i$, we have $\rho_i^\eps \geq R_i^\eps \geq \tilde r R_i\geq \tilde r \underline{a} \rho_i$. Choosing $\tilde r =2/\underline{a}$, we can use the boundedness of $r\mapsto \frac{1+r^2}{(r-\rho_i)^2}$ on $[2\rho_i,\infty)$ and Proposition~\ref{prop.rhoE1} to state
\[
\int_{R_i^\eps \geq \tilde r R_i } (1 + \rho_i^{\eps 2} ) \leq \int_{\rho_i^\eps \geq 2 \rho_i } (1 + \rho_i^{\eps 2} ) \leq C \| \rho_i^\eps -\rho_i\|_{L^2}^2 \leq C \Big(\eps^{2} \EE_1^\eps(t) \Big)^{2\beta_{1,i}(2)}+C\Big(\eps^{2} \EE_1^\eps(t) \Big)^{2\beta_{2,i}(2)}
\]
which ends this proof.
\end{proof}

\section{Uniform estimates and weak convergence of the divergence to zero}\label{convR}


In this section, we derive uniform estimates for $\EE_1^\varepsilon + \EE_2^\varepsilon$ using a Gronwall argument. These bounds allow us to prove the strong convergence of the densities $\rho_\pm^\eps$ in $L^\infty(0,T;L^p(\Omega))$ for any $p<\min (\gamma_+,\gamma_-)$ and the convergence of the partial densities $R_\pm^\varepsilon$ in $C^0([0,T];H^{-1}(\Omega))$. Finally, this enables us to conclude that $\div u^\eps$ converges weakly to 0 in $L^2((0,T)\times \Omega)$.

We work under the hypotheses of the main theorem~\ref{mainth}, in particular, we will always denote by $(\rho^{\epsilon}_+,\rho^{\epsilon}_-,\alpha^{\epsilon}_+,u^{\epsilon})$ a weak solution of System~\eqref{system_eps} in the sense of Definition~\ref{defp} and $(\rho_+,\rho_-,\alpha_+,u)$ a solution of System~\eqref{limitsys}, verifying the regularity properties of Theorem~\ref{sol-lim}.

\subsection{Uniform bound for the relative entropies}\label{sect4.1}

We use Proposition~\ref{prop.estRL2} to write for $i=+,-$
\begin{align*}
 \| R_i^\eps - R_i\|_{L^2} \| \nabla (u^\eps -u)\|_{L^2} &\leq \frac{1}{2\delta_1} \| R_i^\eps - R_i \|_{L^2}^2 + \frac{\delta_1}{2} \lVert \nabla (u-u^\eps) \rVert_{L^2}^2 \\
 &\leq \frac{C}{\delta_1}\left( (\eps^2 \EE^\eps_1)^{\beta_1}+(\eps^2 \EE^\eps_1)^{\beta_2} +\EE_2^\eps \right) + \frac{\delta_1}{2} \lVert \nabla ( u-u^\eps) \rVert_{L^2}^2.
\end{align*}
with an arbitrary $\delta_1>0$, where $\beta_1, \beta_2>0$.
Therefore, using Proposition~\ref{proprelR}, we obtain for $\beta=2\min(\beta_1,\beta_2)>0$
\begin{equation}
 \label{dtE2}
 \EE^\eps_{2}(t) - \EE^\eps_{2}(0) \leq C(\delta_1) \int_0^t (\EE^\eps_1 +\EE^\eps_{2})
 + \frac{\delta_1}{2} \int_0^t \lVert \nabla ( u-u^\eps) \rVert_{L^2(\Omega)}^2 + C(\delta_1)\eps^\beta,
\end{equation}
for all $t\leq T_\eps$ where 
\[
T_\eps:=\sup_{\tilde T \leq T}\{\tilde T , \ \eps^2\EE^\eps_1(t)\leq 1 \text{ on }[0,\tilde T]\}.
\]

In the case $\beta_i<1$, for $i=+,-$, we use the estimate 
\[
\eps^{2\beta_i} (\EE^\eps_1(t))^{\beta_i} \leq C \eps^{2\beta_i} \big(\EE^\eps_1(t)+1\big)
\]
valid for all $t\leq T_\eps$, which allows us to recover a term in $\EE_1^\eps$ and conclude.\\

We now apply the same strategy to $\EE^\eps_1$. More precisely, recall that Proposition~\ref{prop2} gives
\begin{align}
 \EE^\eps_1(t) - \EE^\eps_1(0) + \mu \int_0^t\!\!\! \int_{\Omega} \left( \nabla (u^{\epsilon}-u) \right)^2 +& (\lambda+\mu) \int_0^t\!\!\! \int_{\Omega} \left( \div (u^{\epsilon}-u) \right)^2 \nonumber\\
 \leq& \mu \int_0^t\!\!\! \int_\Omega \frac{\rho^\eps - \rho}{\rho} \Delta u\cdot (u-u^\eps) \label{estim1} \\
 &+ \int_0^t\!\!\! \int_{\Omega} \rho^\eps (u-u^\eps) \cdot \nabla u \cdot (u^\eps-u) \label{estim2} \\
 &+ \int_0^t\!\!\! \int_{\Omega} \frac{\rho-\rho^\eps}{\rho} \nabla \Pi \cdot (u-u^\eps). \label{estim3}\\
 &- \int_0^t\!\!\! \int_{\Omega} \nabla \Pi \cdot (u-u^\eps). \label{estim4}
\end{align}

We deduce first an estimate for the term \eqref{estim2}
\begin{equation*}
 \left\lvert \int_{\Omega} \rho^\eps (u-u^\eps) \cdot \nabla u \cdot (u^\eps-u) \right\rvert \leq \lVert \nabla u \rVert_{L^\infty} \int_{\Omega} \rho^\eps \left\lvert u^\eps -u \right\rvert^2
 \leq 2\lVert \nabla u \rVert_{L^\infty} \EE^\eps_1.
\end{equation*}

Now we focus on the term \eqref{estim1}. Recalling that $\rho^\eps-\rho = \sum_{i=+,-}R_i^\eps - R_i$, we use Hölder inequality ($1/2+1/3+1/6=1$), to get
\begin{equation*}
 \left\lvert \mu \int_{\Omega} \frac{R_\pm^\eps - R_\pm}{\rho} \Delta u (u-u^\eps) \right\rvert \leq \mu \lVert \frac{1}{\rho} \Delta u \rVert_{L^3} \lVert u-u^\eps \rVert_{L^6} \lVert R_\pm^\eps - R_\pm \rVert_{L^2}.
\end{equation*}
Therefore, using Young’s inequality with an arbitrary $\delta_2 >0$, Sobolev embeddings and Poincaré inequality, we obtain from Proposition~\ref{prop.estRL2} 
\begin{align*}
 \left\lvert \mu \int_{\Omega} \frac{R_\pm^\eps - R_\pm}{\rho} \Delta u (u-u^\eps) \right\rvert 
 &\leq \frac{C}{\delta_2} \lVert \frac{1}{\rho} \Delta u \rVert_{L^3(\Omega)}^2\lVert R_\pm^\eps - R_\pm \rVert_{L^2}^2 + C \delta_2 \lVert u-u^\eps \rVert_{L^6(\Omega)}^2\\
& \leq C(\delta_2) (\EE^\eps_1 +\EE^\eps_{2} +\eps^\beta) + \delta_2 C \lVert \nabla (u-u^\eps) \rVert^2_{L^2(\Omega)},
\end{align*}
for all $t\leq T_\eps$.

For \eqref{estim3}, we can apply the same argument where now $\nabla \Pi$ replaces $\mu\Delta u$. We have 
\begin{equation*}
 \int_{\Omega} \left\lvert \frac{\rho-\rho^\eps}{\rho} \nabla \Pi (u-u^\eps) \right\rvert \leq C(\delta_2) (\EE^\eps_1 +\EE^\eps_{2} +\eps^\beta) + \delta_2 C \lVert \nabla (u-u^\eps) \rVert^2_{L^2(\Omega)},
\end{equation*}
for all $t\leq T_\eps$.

 We end by finding an estimate for \eqref{estim4}, using simply the Poincaré inequality:
\begin{equation*}
 \int_{\Omega} \left\lvert \nabla \Pi(u-u^\eps) \right\rvert \leq \frac{1}{\delta_2} \lVert \nabla \Pi \rVert_{L^2}^2 + \delta_2 C \lVert \nabla( u^\eps - u) \rVert_{L^2}^2\leq C(\delta_2)+ \delta_2 C \lVert \nabla( u^\eps - u) \rVert_{L^2}^2.
\end{equation*}

Therefore, we have proved the following inequality for $\EE^\eps_1$:
\begin{multline}\label{dtE1}
 \EE^\eps_1(t) - \EE^\eps_1(0) + \mu \int_0^t\!\!\! \int_{\Omega} \left\lvert \nabla (u^{\epsilon}-u) \right\rvert^2 + (\lambda+\mu) \int_0^t\!\!\! \int_{\Omega} \left\lvert \div (u^{\epsilon}-u) \right\rvert^2 \\
 \leq C(\delta_2) \int_0^t (\EE^\eps_1 +\EE^\eps_{2}) + C\delta_2 \int_0^t \lVert \nabla( u^\eps - u) \rVert_{L^2}^2 + C(\delta_2) 
\end{multline}
where $\delta_2>0$ will be chosen in terms of $\mu$ and $\lambda$. Let us note that the constants in this inequality depends on $\|\Delta u\|_{L^\infty L^3}$, $\|\nabla u\|_{L^\infty L^\infty}$, $\|\nabla \Pi\|_{L^\infty L^3}$ and $T$.

Putting together the estimates for $\EE^\eps_{2}$ and $\EE^\eps_1$, see \eqref{dtE2} and \eqref{dtE1}, we conclude that
\begin{equation}\label{dtE12}
 (\EE^\eps_1+\EE^\eps_2)(t) - (\EE^\eps_1+\EE^\eps_2)(0) + C_{\lambda,\mu} \int_0^t\!\!\! \lVert \nabla( u^\eps - u) \rVert_{L^2}^2 
 \leq C\int_0^t (\EE_1^\eps+\EE^\eps_2) +C,
\end{equation}
for all $t\leq T_\eps$, where $\delta_1$ and $\delta_2$ are chosen in terms of $\lambda$ and $\mu$ in order to get $C_{\lambda,\mu}>0$.
A Gronwall argument yields the uniform bound
\[
\EE^\eps_1(t) + \EE^\eps_2(t) \le C_T, \quad t \in [0,T_\eps],
\]
where $C_T$ is independent of $T_\eps$. Choosing $\eps$ small enough so that $\eps^2 C_T \le \frac12$, we may have
\[
\eps^2 \EE^\eps_1(t) \le \frac12 \quad \text{on } [0,T_\eps].
\]
By definition of $T_\eps$, this implies that $T_\eps = T$, and therefore
\[
\EE^\eps_1(t) + \EE^\eps_2(t) \le C_T, \quad t \in [0,T].
\]

The boundedness of $\EE_1^\eps(t)$ has several direct consequences:
\begin{itemize}
\item By the definition of $\EE_1^\eps$ \eqref{def-E1}, the modulated internal energies,
 \begin{equation*}
 \frac{1}{\varepsilon^2} \alpha_i^\varepsilon H_i(\rho_i^\eps|\rho_i) \text{ is uniformly bounded in } L^\infty(0,T; L^1(\Omega)), \quad i=+,-,
 \end{equation*}
 \item By the inequality \eqref{dtE12}, the velocity gradient,
 \begin{equation*}
 \nabla u^\varepsilon \text{ is uniformly bounded in } L^2(0,T; L^2(\Omega)),
 \end{equation*}
which implies, by the Poincaré inequality in the case of Dirichlet boundary conditions, that the velocity field \(u^\varepsilon\) is uniformly bounded in \(L^2(0,T; H_0^1(\Omega))\).
\item By Proposition~\ref{prop.rhoE1}, we state the strong convergence of the densities $\rho_i^\eps$ to $\rho_i$ in $L^\infty(0,T;L^p(\Omega))$ for all $p<\gamma_i$.
\item By \eqref{est.rhoLgamma}, we finally deduce that $\rho_+^\varepsilon$ is uniformly bounded in $L^\infty(0,T;L^{\gamma_+}(\Omega))$ and $\rho_-^\varepsilon$ is uniformly bounded in $L^\infty(0,T;L^{\gamma_-}(\Omega))$.
\end{itemize}

\subsection{Weak convergence of the divergence to zero}

The previous section also implies that the partial mass densities $R_i^\varepsilon = \alpha_i^\varepsilon \rho_i^\varepsilon$ are uniformly bounded in \(L^\infty(0,T; L^{\gamma_i}(\Omega))\) for \(i=+,-\), because $\alpha_i^\eps\in [0,1]$.

Since $\alpha_\pm^\varepsilon \leq 1$, the sequences $(\alpha_\pm^\varepsilon)_\varepsilon$ are uniformly bounded in $L^\infty((0,T)\times\Omega)$. 
Hence, up to extraction of a subsequence (still denoted by $\alpha_\pm^\varepsilon$), there exist limits $\bar{\alpha}_\pm$ such that
\[
\alpha_\pm^\varepsilon \rightharpoonup \bar{\alpha}_\pm \quad \text{weakly in } L^2(0,T;L^q(\Omega)),
\]
where $1/q+1/p=1/2$, with $p\in (2,\min(\gamma_+,\gamma_-))$ and 
\[
\bar{\alpha}_+ + \bar{\alpha}_- = 1.
\]
Using the product of a strongly convergent sequence and a weakly convergent sequence, we deduce that, for $i=+,-$,
\[
R_i^\varepsilon = \alpha_i^\varepsilon \rho_i^\varepsilon \rightharpoonup \bar{\alpha}_i \rho_i 
\quad \text{weakly in } L^2((0,T)\times \Omega).
\]

We recall that $u^\varepsilon$ is uniformly bounded in $L^2(0,T;H^1_0(\Omega))$ by the previous subsection. 
By Sobolev embeddings in space dimension $d=3$, we have that $u^\eps$ is uniformly bounded in $L^2(0,T;L^6(\Omega))$. 
Since $R_\pm^\varepsilon \in L^\infty(0,T;L^2(\Omega))$, Hölder's inequality implies that $R_\pm^\varepsilon u^\varepsilon$ is uniformly bounded in $L^2(0,T;L^{3/2}(\Omega))$, and consequently, $\partial_t R_\pm^\varepsilon =-\div (R_\pm^\varepsilon u^\varepsilon)$ is uniformly bounded in $L^2(0,T;W^{-1,3/2}(\Omega))$.

Applying the Aubin--Lions lemma together with the compact and continuous embeddings
\[
L^2(\Omega) \hookrightarrow\hookrightarrow H^{-1}(\Omega)
\hookrightarrow W^{-1,\frac{3}{2}}(\Omega),
\]
we infer that, up to the extraction of a subsequence,
\[
R_\pm^\varepsilon \to \bar{R}_\pm
\quad \text{strongly in } C^0([0,T];H^{-1}(\Omega)).
\]
By uniqueness of the limit, we conclude that
\[
\bar{R}_\pm = \bar{\alpha}_\pm \rho_\pm.
\]

Since $u^\eps-u $ is uniformly bounded in $L^2(0,T;H^1_0(\Omega))$, we infer by the Banach–Alaoglu theorem that there exists a function 
$v \in L^2(0,T;H^1_0(\Omega))$ and a subsequence (not relabeled) such that
\[
u^\varepsilon - u \rightharpoonup v
\quad \text{weakly in } L^2(0,T;H^1_0(\Omega)).
\]

Passing to the limit in the mass equations, we obtain, for $i=+,-$,
\[
\partial_t (\bar{\alpha}_i \rho_i) + \operatorname{div}\big( \bar{\alpha}_i \rho_i (u+v) \big) = 0,
\]
in the sense of distribution,
which yields
\[
\partial_t \bar{\alpha}_i + \operatorname{div}\big( \bar{\alpha}_i (u+v) \big) = 0,
\]
because we recall that $\rho_i$ is constant, see the limit system~\eqref{limitsys}.

Summing the two equations for $i=+,-$, and using $\bar{\alpha}_+ + \bar{\alpha}_- = 1$ together with $\operatorname{div} u = 0$, we deduce
\[
\operatorname{div} v = 0.
\]

Therefore, we conclude that
\begin{equation*}
\operatorname{div} u^\varepsilon \rightharpoonup 0 \quad \text{weakly in } L^2((0,T)\times\Omega).
\end{equation*}
By the uniqueness of this limit, we realize that we do not need to extract a subsequence and that this weak convergence holds true for the full sequence $\eps\to 0$.

\section{Strong convergences}\label{End}

We revisit the proof of the estimates \eqref{dtE12}. Keeping \eqref{estim4} unchanged, we observe that combining \eqref{dtE2}--\eqref{estim3} yields 
\[
(\EE^\eps_1+\EE^\eps_2)(t) - (\EE^\eps_1+\EE^\eps_2)(0) + C_{\lambda,\mu} \int_0^t\!\!\! \lVert \nabla( u^\eps - u) \rVert_{L^2}^2 
 \leq C\int_0^t (\EE_1^\eps+\EE^\eps_2) +C\eps^\beta + \Big|\int_0^t\!\!\! \int_{\Omega} \nabla \Pi \cdot (u-u^\eps) \Big|,
\]
for all $t\leq T$. Denoting 
\[
r_\eps(t):=(\EE^\eps_1+\EE^\eps_2)(0) + C\eps^\beta + \Big|\int_0^t\!\!\! \int_{\Omega} \Pi \div(u-u^\eps) \Big|\geq 0
\]
the Gronwall lemma gives
\[
(\EE^\eps_1+\EE^\eps_2)(t)\leq r_\eps(t) + C\int_0^t r_\eps(s) e^{C(t-s)} \dd s \leq r_\eps(t) + Ce^{CT}\int_0^t r_\eps(s) \dd s .
\]
In particular
\[
\int_0^T (\EE^\eps_1+\EE^\eps_2)(t)\dd t\leq C \int_0^T r_\eps(t)\dd t.
\]
For every $t\in [0,T]$, using that $\mathds{1}_{[0,t]}\Pi\in L^2((0,T)\times \Omega)$, the weak convergence of $\div(u-u^\eps)$ to 0 means that $r_\eps(t)\to 0$ pointwise as $\eps\to 0$. Moreover, we have the uniform estimates
\[
r_\eps(t) \leq (\EE^\eps_1+\EE^\eps_2)(0) + C\eps^\beta + \int_0^t\!\!\! \int_{\Omega} |\Pi| | \div(u-u^\eps)| \leq C
\]
because $\div u^\eps$ is uniformly bounded in $L^2((0,T)\times \Omega)$. By the dominated convergence theorem, we have
\[
\int_0^T r_\eps(t)\dd t \to 0
\]
hence $\EE_1^\eps$ and $\EE_2^\eps$ tend to zero in $L^1(0,T)$. As they are uniformly bounded in $L^\infty(0,T)$, we get by interpolation that $\EE_1^\eps$ and $\EE_2^\eps$ tend to zero in $L^q(0,T)$ for any $q\in [1,\infty)$, giving in particular the last convergence stated in Theorem~\ref{mainth}.

The first inequality of this subsection implies
\[
 \| u^\eps - u \|_{L^2(0,T;H^1_0(\Omega))}^2 \leq C \int_0^T (\EE^\eps_1+\EE^\eps_2)(s)\dd s + Cr_\eps(T) \to 0,
\]
which gives the fourth convergence in Theorem~\ref{mainth}.

By Proposition~\ref{prop.estRL2}, we also deduce the strong convergence of $R_i^\eps$ to $R_i^\eps$ in $L^q(0,T;L^2(\Omega))$ for any $q\in [1,\infty)$, which is the third convergence in Theorem~\ref{mainth}.\newline

Concerning the strong convergence of the volume fraction, we use the bi-fluid structure to obtain a control of the volume fractions by the partial densities. 

We first prove the following fundamental result:
\begin{lemma}\label{lemma_alpha}
Assume that there exists a constant $c>0$ such that $R_+^\eps > c$. Therefore, there exists $C>0$ such that
\begin{equation*}
 \lvert \alpha_+^\eps - \alpha_+ \rvert \leq C \lvert R_+^\eps - R_+ \rvert + C \lvert R_-^\eps - R_- \rvert.
\end{equation*}
\end{lemma}

By symmetry, we have the similar result for $|\alpha_-^\eps - \alpha_- |$ when $R_-^\eps > c$. To prove this result, we first establish a key lemma, which provides the necessary technical tools.

\begin{lemma}\label{technicalemma}
Let $\gamma >0$. For all $d_1, d_2 >0$ such that $d_2 \in [\underline{c},\overline{c}]\subset \R^+_*$, there exists a unique solution $(a_1,a_2) \in (0,1)$ that satisfies
\begin{align*}
 d_1 a_1^\gamma + (a_1-1) = 0,\\
 d_2 a_2^\gamma + (a_2-1) = 0.
\end{align*}
Moreover, there exists a constant $C$ (that only depends on $\gamma, \overline{c}, \underline{c}$) such that
\begin{equation*}
 \lvert a_1 - a_2 \rvert \leq C \lvert d_1 - d_2 \rvert.
\end{equation*}
\end{lemma}

We start by proving Lemma~\ref{technicalemma}.

\begin{proof}[Proof of Lemma~\ref{technicalemma}]
For each $i =1,2$, we define the function $F_{d_i} : \mathbb{R}_+ \to \mathbb{R}$ by
\[
F_{d_i}(x) := d_i x^\gamma + x - 1.
\]
The function $F_{d_i}$ is strictly increasing with respect to $x$. Moreover, we have
$F_{d_i}(0) = -1$ and $F_{d_i}(1) = d_i > 0$.
By continuity, there exists a unique solution $x_i \in (0,1)$ to the equation
$F_{d_i}(x) = 0$.
We denote this solution by $a_i$ for $i = 1,2$.

We then define the function $g:[\underline{c},\overline{c}]\to (0,1)$ by setting $a = g(d)$ such that $F_d(a)=0$. The monotonicity of $F_{d}$ with respect to $x$ implies that $g$ is decreasing with respect to $d$.

First, consider the case where $d_1 < \underline{c}/2$ or $d_1 > 2\overline{c}$. Since $d_2 \in [\underline{c}, \overline{c}]$, there exists a constant $C > 0$ sufficiently large such that
\[
|d_1 - d_2| \geq 1/C.
\]
Moreover, since $a_1, a_2 \in (0,1)$, it follows that
\[
|a_1 - a_2| \leq 1 \leq C |d_1 - d_2|.
\]

Now consider the case where $d_1 \in [\underline{c}/2,2 \overline{c}]$. Since the function $g$ is decreasing, it follows that $a_1 \in [g(2 \overline{c}), g(\frac{\underline{c}}{2})]$. We then introduce the compact set $K = [\frac{\underline{c}}{2},2 \overline{c}] \times [g(2 \overline{c}), g(\frac{\underline{c}}{2})]$ and and define the function $F: K \to \mathbb{R}$ by $F(d,x) = d x^\gamma +x-1$. The function $F$ is of class $C^\infty$ on $K$.

For any $(d_0,x_0) \in K$ satisfying $F(d_0,x_0)=0$, as $\partial_x F(d_0,x_0) \neq 0$, the Implicit Function Theorem ensures the existence of an open neighbourhood $U_{0,\mathrm{loc}}$ of $d_0$ and a unique function $g_{0,\mathrm{loc}} : U_{0,\mathrm{loc}} \to \mathbb{R}$ such that
\[
x = g_{0,\mathrm{loc}}(d)
\quad \text{and} \quad
F\bigl(d, g_{0,\mathrm{loc}}(d)\bigr)=0
\quad \text{for all } d \in U_{0,\mathrm{loc}}.
\]
Since the solution $(d,x)$ of the equation $F(d,x)=0$ is unique, it follows that
\[
g_{0,\mathrm{loc}} = g.
\]

From the Implicit Function Theorem, we also have
\[
g_{0,\mathrm{loc}}'(d) = -\,\frac{\partial_d F(d,g(d))}{\partial_x F(d,g(d))} = -\,\frac{x^\gamma}{\gamma d x^{\gamma-1}+1}, \qquad \text{with } x = g(d).
\]
Since $(d,x) \in K$, there exists a constant $L>0$ such that
\[
\left| g'(d) \right| \le L \quad \text{for all } d \in \Bigl[\frac{\underline{c}}{2},2 \overline{c}\Bigr].
\]
It then follows that
\begin{equation*}
 |a_1 - a_2| = |g(d_1) - g(d_2)| \le L\, |d_1 - d_2|.
\end{equation*}

Hence, we have shown that there exists a constant $C>0$ such that
\[
|a_1 - a_2| \le C\, |d_1 - d_2|,
\]
which concludes the proof.

\end{proof}

Having established Lemma~\ref{technicalemma}, we can now complete the proof of Lemma~\ref{lemma_alpha}.

\begin{proof}[Proof of Lemma~\ref{lemma_alpha}.]
From the equality of pressures we have
\begin{equation*}
 \frac{R_+^{\gp}}{\ap^{\gp}} = \frac{R_-^{\gm}}{(1-\ap)^{\gm}},
\end{equation*}
which gives the following equation in $\ap$, for $\gamma=\gp/\gm$:
\begin{equation}\label{eqd1}
 \frac{R_-}{R_+^{\gamma}} \ap^\gamma + (\ap-1)=0.
\end{equation}

Similarly, we obtain for $\ap^\eps$:
\begin{equation}\label{eqd2}
 \frac{R_-^\eps}{(R_+^\eps)^\gamma} (\ap^\eps)^\gamma + (\ap^\eps-1) = 0.
\end{equation}

If we denote by $d_1 = \frac{R_-^\eps}{(R_+^\eps)^\gamma}$ and $d_2 = \frac{R_-}{R_+^{\gamma}}$, the hypothesis $d_2 \in [\underline{c},\overline{c}]$ is verified because $\alpha_\pm\in [\underline{a},\overline{a}]\subset (0,1)$, see Remark~\ref{rkalpha}, and we can apply Lemma~\ref{technicalemma}. It follows that we obtain a unique solution $(\ap^\eps,\ap) \in (0,1)$ such that Equations \eqref{eqd1}-\eqref{eqd2} are satisfied and that there exists a constant $C$ such that
\begin{equation*}
 | \ap^\eps - \ap | \leq C \left\lvert \frac{R_-^\eps}{(R_+^\eps)^\gamma} - \frac{R_-}{R_+^{\gamma}} \right\rvert.
\end{equation*}

By adding and subtracting $\frac{R_-}{(R_+^\eps)^{\gamma}}$ and using the assumption $R_+^\eps > c$, we obtain
\begin{equation*}
 | \ap^\eps - \ap | \leq C \left\lvert R_-^\eps - R_- \right\rvert + C R_- \left\lvert \frac{1}{(R_+^\eps)^\gamma} - \frac{1}{R_+^{\gamma}} \right\rvert.
\end{equation*}

Finally, using the mean value theorem for the function $f(x) = \frac{1}{x^\gamma}$ defined for $x \in [\min(c,R_{+,0}),+\infty)$, we can conclude the proof.
\end{proof}

As a consequence of Lemma~\ref{lemma_alpha}, we want to show that
\begin{equation}\label{2ndterm}
 \int_\Omega \rho_\pm^2 \lvert\alpha_\pm^\eps-\alpha_\pm \rvert^2 \leq C\eps^2 \EE^\eps_1 + C \sum_{i=+,-} \int_\Omega \lvert R_i^\eps - R_i \rvert.
\end{equation} 
Let us consider $\tilde{C_0}>0$, depending only on $\overline{a}$ and $C_0$ such that $R_\pm \leq \tilde{C_0}$.
\begin{itemize}
 \item In the set $\Omega_{1,\pm}:=\{x, \ (R_\pm^\eps + R_\pm) \leq 2\tilde C_0\}$, we use the equality:
\begin{equation*}
 \rho_\pm^2 \left(\alpha_\pm^\eps-\alpha_\pm \right)^2 = (\alpha_\pm^\eps)^2 \left(\rho_\pm-\rho_\pm^\eps\right)^2 + 2 \alpha_\pm^\eps \rho_\pm \left( R_\pm^\eps - R_\pm \right)+ \left(R_\pm^2 - (R_\pm^\eps)^2\right).
\end{equation*}

For the last term, we use $(R_\pm^\eps + R_\pm) \leq 2\tilde C_0$ to get:
\begin{equation*}
 \int_{\Omega_{1,\pm}} \lvert (R_\pm^\eps)^2 - R_\pm^2 \rvert \leq \int_{\Omega_{1,\pm}} \lvert R_\pm^\eps - R_\pm \rvert (R_\pm^\eps + R_\pm) \leq C \int_\Omega \lvert R_\pm^\eps - R_\pm \rvert.
\end{equation*}
Therefore in the first set we obtain
\begin{align*}
 \int_{\Omega_{1,\pm}} \rho_\pm^2 \lvert\alpha_\pm^\eps-\alpha_\pm \rvert^2 &\leq \int_\Omega \alpha_\pm^\eps \lvert \rho_\pm^\eps - \rho_\pm \rvert^2 + 2 \lVert \rho_\pm \rVert_{L^\infty} \int_\Omega \lvert R_\pm^\eps - R_\pm \rvert + C \int_\Omega \lvert R_\pm^\eps - R_\pm \rvert\nonumber \\
 &\leq \int_\Omega \alpha_\pm^\eps \lvert \rho_\pm^\eps - \rho_\pm \rvert^2+ C \int_\Omega \lvert R_\pm^\eps - R_\pm \rvert,
\end{align*}
which is controlled by the relative entropy $\EE^\eps_1$ \eqref{def-E1} through Inequality \eqref{eqB}:
\begin{equation*}
 \int_\Omega \alpha_\pm^\eps \lvert \rho_\pm^\eps - \rho_\pm \rvert^2 \leq C\int_{\Omega}\alpha_\pm^\eps H_\pm(\rho_\pm^\eps\vert \rho_\pm) 
 \leq C\eps^2 \EE^\eps_1.
\end{equation*}
where we have used $\gp, \gm \geq 2$.

\item Now in the complementary set $\Omega_{2,\pm} := \{x,\ (R_\pm^\eps + R_\pm) > 2\tilde C_0\}$, we have $R_\pm^\eps>\tilde C_0$ and the result comes directly from Lemma~\ref{lemma_alpha}:
\begin{equation*}
 \int_{\Omega_{2,\pm}} \rho_\pm^2 \lvert\alpha_\pm^\eps-\alpha_\pm \rvert^2 \leq 2 \lVert \rho_\pm \rVert_{L^\infty}^2 \int_\Omega \Big(C \lvert R_+^\eps - R_+ \rvert + C \lvert R_-^\eps - R_- \rvert\Big),
\end{equation*}
\end{itemize}
where we also have used $\lvert\alpha_\pm^\eps-\alpha_\pm \rvert \leq 2$.
This ends the proof of \eqref{2ndterm}.

By the uniform estimates of $\EE_1^\eps(t)$, we finish the proof of Theorem~\ref{mainth} by noticing that \eqref{2ndterm} reads
\[
\| (\alpha_i^\eps - \alpha_i)(t,\cdot)\|_{L^2(\Omega)}^2 \leq C\eps^2 + C \sum_{i=+,-}\| (R_i^\eps - R_i)(t,\cdot)\|_{L^2(\Omega)}\to 0 \quad \text{in }L^q(0,T)
\]
for any $q\in [1,\infty)$.

\bigskip

\noindent {\bf Acknowledgements.} The author would like to thank Didier Bresch and Christophe Lacave for their many tips on writing this text. 
She also gratefully acknowledges the University Savoie Mont Blanc and the Agence Nationale pour la Recherche (ANR) for her PhD position fellowship linked to the CPJ Anamod managed by Christophe Lacave. She also wants to thank for some support of the ANR under France 2030 bearing the reference ANR-23-EXMA-004 (Complexflows project) related to the PEPR Maths-ViVEs. The author wants also to gratefully acknowledge the partial support by the ANR grant ANR-23-CE40-0014-01 (ANR Bourgeons).

\begin{appendices}

\section{Entropy computations}\label{app1}

In this appendix, we derive entropy relations for general velocity fields $u$, extending the computations presented in the main text (see Section~\ref{computationsE1}) and providing a framework that could be applied to stability analyses or weak-strong uniqueness studies. For that, we consider two solutions of System~\eqref{system} (without nondimensionalization and the Mach number): a weak solution still denoted by indices $\eps$, $(\rho^{\epsilon}_+,\rho^{\epsilon}_-,\alpha^{\epsilon}_+,u^{\epsilon})$, and a sufficiently regular solution $(\rho_+,\rho_-,\alpha_+,u)$. We define the following quantities:
\begin{itemize}
 \item the energy
 \begin{equation*}
 E(t) := \frac{1}{2} \int_{\Omega} \Big(\rho^\varepsilon |u^\eps|^2\Big)(t,x) \dd x
 + \sum_{i=+,-} \int_{\Omega} \frac{\alpha_i^{\epsilon}(\rho_i^{\epsilon})^{\gamma_i}(t,x)}{\gamma_i-1} \dd x ,
 \end{equation*}
 \item the modulated entropy
 \begin{equation*}
 \begin{aligned}
 \EE_1(t) &= \EE_1\Big(\rho^{\epsilon}_+,\rho^{\epsilon}_-,\alpha^{\epsilon}_+,u^{\epsilon}| \rho_+,\rho_-,\alpha_+,u\Big) \\
 &:=
 \frac{1}{2} \int_{\Omega} \rho^{\epsilon}(t,x) |u^{\epsilon} - u|^2(t,x)\dd x 
 + \sum_{i=+,-} \int_{\Omega} \alpha_i^{\epsilon}(t,x) H(\rho_i^\eps | \rho_i)(t,x)\dd x .
 \end{aligned}
 \end{equation*}
\end{itemize}

\begin{proposition}\label{prop1app} If $(\rho^{\epsilon}_+,\rho^{\epsilon}_-,\alpha^{\epsilon}_+,u^{\epsilon})$ is a weak solution of System~\eqref{system}, defined as in Definition~\ref{defp}, then for all $(\rho_+,\rho_-,\alpha_+,u)$ sufficiently regular solutions of the same System~\eqref{system}, we have
 \begin{align*}
 \EE_1(t) - \EE_1(0) \leq& - \mu \int_0^t\!\!\! \int_{\Omega} \left( \nabla (u^{\epsilon}-u) \right)^2 - (\lambda+\mu) \int_0^t\!\!\! \int_{\Omega} \left( \div (u^{\epsilon}-u) \right)^2 \dd x \dd s \\
 &+ \int_0^t\!\!\! \int_{\Omega} \left(\alpha_+^{\varepsilon} \rho_{+}^{\varepsilon}+\alpha_-^{\varepsilon} \rho_-^{\epsilon}\right) (\partial_t u + u \cdot \nabla u)\cdot (u-u^{\epsilon}) \dd x \dd s \\
 &+ \int_0^t\!\!\! \int_{\Omega} \left(\alpha_+^{\varepsilon} \rho_{+}^{\varepsilon}+\alpha_-^{\varepsilon} \rho_-^{\epsilon}\right) (u^{\epsilon}-u)\cdot \nabla u \cdot (u-u^{\epsilon}) \dd x \dd s \\
 &+\int_0^t\!\!\! \int_{\Omega} \Big(\alpha_+^\eps p_+(\rho_+) + \alpha_-^\eps p_-(\rho_-) - \left(\alpha_+^{\varepsilon} p_+\left(\rho_{+}^{\varepsilon}\right)+\alpha_-^{\varepsilon} p_-\left(\rho_-^{\varepsilon}\right)\right) \Big) \div u \dd x \dd s \\
 &+ \sum_{i=+,-} \int_0^t\!\!\! \int_{\Omega} \gamma_i \alpha_i^{\epsilon} \rho_i^{\gamma_i-2}(\rho_i - \rho_i^{\epsilon})\partial_t \rho_i \dd x \dd s\\
 &+ \sum_{i=+,-} \int_{\Omega} \gamma_i \rho_i^{\gamma_i-2} (\alpha_i\rho_iu-\alpha_i^{\epsilon} \rho_i^{\epsilon} u^{\epsilon}) \cdot \nabla\rho_i \dd x \dd s \\
 &-\mu \int_0^t\!\!\! \int_{\Omega} \nabla u : \nabla(u^\eps-u) - (\lambda+\mu) \dd x\dd s \int_0^t\!\!\! \int_{\Omega} \div u \div(u^\eps-u) \dd x\dd s.
 \end{align*}
\end{proposition} 

\begin{remark}
The equation satisfied by the strong solution $(\rho_\pm, \alpha_+,u)$ is not written explicitly in compact form in the right-hand side above. Instead, its contribution appears through the various decomposed terms, which correspond to the expansion of the momentum equation.
\end{remark}

\begin{proof}
We have
\begin{equation}\label{relentropy}
 \begin{aligned}
 \EE_1(t) - \EE_1(0) =& E(t) - E(0) \\
 &+\frac{1}{2} \int_{\Omega} \Big(\rho^\eps |u|^2\Big)(t,\cdot) - \frac{1}{2} \int_{\Omega} \Big(\rho^\eps |u|^2\Big)(0,\cdot) \\
 &- \int_{\Omega} \Big(\rho^\eps u^\eps \cdot u\Big)(t,\cdot) + \int_{\Omega} \Big(\rho^\eps u^\eps \cdot u\Big)(0,\cdot) \\
 &-\sum_{i=+,-} \int_{\Omega} \frac{\alpha_i^{\epsilon}}{(\gamma_i -1)} \Big(\rho_i^{\gamma_i} + \gamma_i \rho_i^{\gamma_i-1}(\rho_i^{\epsilon}-\rho_i) \Big)(t,\cdot)\\
 & +\sum_{i=+,-} \int_{\Omega} \frac{\alpha_i^{\epsilon}}{(\gamma_i -1)} \Big(\rho_i^{\gamma_i} + \gamma_i \rho_i^{\gamma_i-1}(\rho_i^{\epsilon}-\rho_i) \Big)(0,\cdot).
 \end{aligned}
\end{equation}

Let us consider mass equations with $\phi=\frac{|u|^2}{2}$ and add the two, we get:
\begin{align*}
 \frac{1}{2} \int_{\Omega} \Big(\rho^\eps |u|^2\Big)(t,\cdot) - \frac{1}{2} \int_{\Omega} \Big(\rho^\eps |u|^2\Big)(0,\cdot) 
 &=\frac{1}{2} \int_0^t\!\!\!\int_{\Omega} \rho^\eps \left(\partial_t|u|^2+u^\eps \cdot \nabla |u|^2\right) \nonumber \\
&= \int_0^t \!\!\!\int_\Omega \rho^\eps u\cdot \Big(\partial_t u + (u^{\epsilon} \cdot \nabla) u\Big).
\end{align*}

Let us now use the momentum equation by taking $\phi=u$, we get:
\begin{multline*}
 -\int_{\Omega} \Big(\rho^\eps u^\eps \cdot u\Big)(t,\cdot) + \int_{\Omega} \Big(\rho^\eps u^\eps \cdot u\Big)(0,\cdot) = -\int_0^t\!\!\!\int_{\Omega}\rho^\eps u^{\varepsilon} \cdot (\partial_t u + (u^{\varepsilon} \cdot \nabla) u )
 \\
 - \int_0^t\!\!\! \int_{\Omega} \rho_{+}^{\varepsilon \gamma_+}\div u 
 +\mu \int_0^t\!\!\!\int_{\Omega} \nabla u^{\varepsilon} : \nabla u + (\mu+\lambda) \int_0^t\!\!\!\int_{\Omega} \left(\operatorname{div} u^{\varepsilon}\right)(\operatorname{div} u),
\end{multline*}
where we can replace
\begin{align*}
- \int_0^t\!\!\! \int_{\Omega} \rho_{+}^{\varepsilon \gamma_+}\div u 
&=
- \int_0^t\!\!\! \int_{\Omega}(\alpha_+^\eps+\alpha_-^\eps) \rho_{+}^{\varepsilon \gamma_+}\div u \\
& = - \int_0^t\!\!\! \int_{\Omega}(\alpha_+^\eps p_+ (\rho_+^\eps) +\alpha_-^\eps p_-(\rho_-^\eps) )\div u .
\end{align*}

We can also rewrite 
\begin{align*}
 \int_{\Omega} \rho^\eps (\partial_t u + (u^{\epsilon} \cdot \nabla) u)\cdot (u-u^{\epsilon})
 =& \int_{\Omega} \rho^\eps (\partial_t u + (u \cdot \nabla) u)\cdot (u-u^{\epsilon}) \\
 &+ \rho^\eps \Big(((u^{\epsilon}-u)\cdot \nabla) u \Big)\cdot (u-u^{\epsilon}).
\end{align*}

By the equality of the pressures and the fact that $\alpha_+^\eps+\alpha_-^\eps=1$, we have
\begin{equation*}
 -\sum_{i=+,-} \int_{\Omega} \frac{\alpha_i^{\epsilon}}{(\gamma_i -1)} \Big(\rho_i^{\gamma_i} - \gamma_i \rho_i^{\gamma_i-1}\rho_i \Big)
 = \sum_{i=+,-}\int_{\Omega} \alpha_i^{\epsilon}\rho_i^{\gamma_i} 
 = \int_{\Omega} \rho_+^{\gamma_+} 
\end{equation*}
hence
\begin{multline*}
 - \sum_{i=+,-} \int_{\Omega} \frac{\alpha_i^{\epsilon}}{(\gamma_i -1)} \Big(\rho_i^{\gamma_i} - \gamma_i \rho_i^{\gamma_i-1}\rho_i \Big)(t,\cdot) + \int_{\Omega} \frac{\alpha_i^{\epsilon}}{(\gamma_i -1)} \Big(\rho_i^{\gamma_i} - \gamma_i \rho_i^{\gamma_i-1}\rho_i) \Big)(0,\cdot) \\
 =\int_0^t\!\!\!\int_\Omega (\alpha_+^\eps+\alpha_-^\eps)\partial_t \rho_+^{\gamma_+} 
 =\sum_{i=+,-}\int_0^t\!\!\!\int_\Omega \alpha_i^\eps \gamma_i \rho_i^{\gamma-1} \partial_t \rho_i 
\end{multline*}

For $i=+,-$, remark now that the mass equation with $\phi = \rho_i^{\gamma_i-1}$ gives
\begin{align*}
 -\int_{\Omega} \frac{\gamma_i \alpha_i^{\epsilon}}{(\gamma_i -1)} \rho_i^{\gamma_i-1}\rho_i^{\epsilon} (t,\cdot) +\int_{\Omega} \frac{\gamma_i\alpha_i^{\epsilon}}{(\gamma_i -1)} \rho_i^{\gamma_i-1}\rho_i^{\epsilon} (0,\cdot)
 & =- \int_0^t\!\!\!\int_\Omega \gamma_i\alpha_{i}^\eps\rho_{i}^\eps\rho_i^{\gamma_i-2} \Big(\partial_t \rho_i + (u^{\epsilon} \cdot \nabla) \rho_i\Big).
\end{align*}

We then have obtained
\begin{align*}
 -\sum_{i=+,-} \int_{\Omega} \frac{\alpha_i^{\epsilon}}{(\gamma_i -1)} \Big(\rho_i^{\gamma_i} + \gamma_i \rho_i^{\gamma_i-1}(\rho_i^{\epsilon}-\rho_i) \Big)(t,\cdot)&\\ 
 +\sum_{i=+,-} \int_{\Omega} \frac{\alpha_i^{\epsilon}}{(\gamma_i -1)} \Big(\rho_i^{\gamma_i} + \gamma_i \rho_i^{\gamma_i-1}(\rho_i^{\epsilon}-\rho_i) \Big)(0,\cdot)
 &=
 \sum_{i=+,-} \int_0^t\!\!\! \int_{\Omega} \gamma_i \alpha_i^{\epsilon} \rho_i^{\gamma_i-2}(\rho_i - \rho_i^{\epsilon})\partial_t \rho_i \\
 &+ \sum_{i=+,-} \int_0^t\!\!\!\int_{\Omega} \gamma_i \rho_i^{\gamma_i-2} (\alpha_i\rho_i u-\alpha_i^{\epsilon} \rho_i^{\epsilon} u^{\epsilon}) \cdot \nabla\rho_i \\
 &- \sum_{i=+,-} \int_0^t\!\!\!\int_{\Omega} \gamma_i \rho_i^{\gamma_i-2} \alpha_i\rho_i u \cdot \nabla\rho_i
\end{align*}
Using again that $p_+(\rho_+) = p_-(\rho_-)$ and $\alpha_++\alpha_-=1$, we can rewrite the last term
\begin{align*}
 - \sum_{i=+,-} \int_0^t\!\!\!\int_{\Omega} \gamma_i \rho_i^{\gamma_i-2} \alpha_i\rho_i u \cdot \nabla\rho_i
 =&- \sum_{i=+,-} \int_0^t\!\!\!\int_{\Omega} \alpha_i u \cdot \nabla p_i(\rho_i) 
 = -\int_0^t\!\!\!\int_{\Omega} u \cdot \nabla p_+(\rho_+)\\
 =& \int_0^t\!\!\!\int_{\Omega} p_+(\rho_+)\div u \\
 =& \int_0^t\!\!\!\int_{\Omega} (\alpha_+^\eps p_+(\rho_+)+\alpha_-^\eps p_-(\rho_-)) \div u
\end{align*}
where we have used that $u$ satisfies the Dirichlet boundary condition.

Therefore, using the energy estimate \eqref{energy_estimate} satisfied by $(\rho_+^{\epsilon}, \rho_-^{\epsilon}, \alpha_+^{\epsilon},u^{\epsilon}) $ and plugging all the previous computations into \eqref{relentropy}, we obtain:
\begin{align*}
 \EE_1(t) - \EE_1(0) &\leq -\mu \int_0^t\!\!\! \int_{\Omega}\left( \nabla u^\eps \right)^2 \dd x \dd t - \left(\lambda + \mu\right) \int_0^t\!\!\! \int_{\Omega} (\div u^\eps )^2 \dd x \dd t\\
 &+\int_0^t\!\!\! \int_{\Omega} \left(\alpha_+^{\varepsilon} \rho_{+}^{\varepsilon}+\alpha_-^{\varepsilon} \rho_-^{\epsilon}\right) (\partial_t u + u \cdot \nabla u)\cdot (u-u^{\epsilon}) \\
 &+ \int_0^t\!\!\! \int_{\Omega} \left(\alpha_+^{\varepsilon} \rho_{+}^{\varepsilon}+\alpha_-^{\varepsilon} \rho_-^{\epsilon}\right) (u^{\epsilon}-u)\cdot \nabla u \cdot (u-u^{\epsilon})\\
 &+ \int_0^t\!\!\! \int_{\Omega} \Big(\alpha_+^\eps p_+(\rho_+) + \alpha_-^\eps p_-(\rho_-) - \left(\alpha_+^{\varepsilon} p_+\left(\rho_{+}^{\varepsilon}\right)+\alpha_-^{\varepsilon} p_-\left(\rho_-^{\varepsilon}\right)\right) \Big) \div u \\
 &+ \sum_{i=+,-} \int_0^t\!\!\! \int_{\Omega} \gamma_i \alpha_i^{\epsilon} \rho_i^{\gamma_i-2}(\rho_i - \rho_i^{\epsilon})\partial_t \rho_i + \sum_{i=+,-} \int_{\Omega} \gamma_i \rho_i^{\gamma_i-2} (\alpha_i\rho_i-\alpha_i^{\epsilon} \rho_i^{\epsilon} u^{\epsilon}) \cdot \nabla\rho_i \\
 &+\mu \int_0^t\!\!\! \int_{\Omega} \nabla u^{\varepsilon} : \nabla u + (\mu+\lambda) \int_0^t\!\!\! \int_{\Omega} \left(\operatorname{div} u^{\varepsilon}\right)(\operatorname{div} u).
\end{align*}

Finally, we end the proof by noticing that by virtue of monotonicity of the operators $\nabla$ and $\div$, namely, we write
\begin{equation*}
 - \nabla u^{\varepsilon} : \nabla (u^\eps-u) = - \left(\nabla (u^\eps - u) \right)^2 - \nabla u : \nabla(u^\eps - u),
\end{equation*}
and
\begin{equation*}
 - \div u^\eps \div(u^\eps-u) = - \left( \div(u^\eps - u) \right)^2 - \div u \div(u^\eps - u).
\end{equation*}
\end{proof}

\end{appendices}


\adrese{}

\end{document}